\documentclass[12pt]{elsarticle}

\usepackage{paralist, mathdots} 
\usepackage{amssymb} 
\usepackage{amsmath,tikz,wasysym,multirow} 
\usepackage{xfrac}
\usepackage{graphicx}
\usepackage{url}

\newtheorem{proposition}{Proposition}[section]
\newtheorem{theorem}{Theorem}[section]
\newtheorem{example}{Example}[section]

\newtheorem{remark}{Remark}[section]
\newtheorem{definition}{Definition}[section]
\newtheorem{corollary}{Corollary}[section]
\newtheorem{lemma}{Lemma}[section]
\newtheorem{conjecture}{Conjecture}[section]
\newenvironment{proof}{\textbf{Proof.}}{\qquad $\Box$ \bigskip }

\newcommand{\R}{\mathbb{R}}
\newcommand{\C}{\mathbb{C}}
\newcommand{\mc}{\mathcal}
\newcommand{\ii}{\mathrm{i}}
\newcommand{\F}{\mathcal{F}}

    \definecolor{helena}{rgb}{.2,.8,.4}
    \definecolor{steve}{rgb}{.8,.2,.2}
    \definecolor{todo}{rgb}{.2,.2,.8}

\begin{document}
\title{Stochastic Matrices Realising the Boundary of the Karpelevi\v c Region}

\author{Stephen Kirkland \fnref{fn1}}
	\ead{stephen.kirkland@umanitoba.ca}
\author{Helena \v Smigoc \corref{cor1}%
	\fnref{fn2}}
	\ead{helena.smigoc@ucd.ie}



\cortext[cor1]{Corresponding author}

\fntext[fn1]{Department of Mathematics, University of Manitoba, Winnipeg, MB, Canada.}

\fntext[fn2]{School of Mathematics and Statistics, University College Dublin, Belfield, Dublin 4, Ireland.} 

\begin{abstract} A celebrated result of Karpelevi\v c describes $\Theta_n,$ the collection of all eigenvalues arising from the stochastic matrices of order $n.$ The boundary of $\Theta_n$ consists of roots of certain one-parameter families of polynomials, and those polynomials are naturally associated with the so--called reduced Ito polynomials of Types 0, I, II and III. 
	
	In this paper we explicitly characterise all $n \times n$ stochastic matrices whose characteristic polynomials are of Type 0 or Type I, and all sparsest stochastic matrices of order $n$  whose characteristic polynomials are of Type II or Type III. The results provide insights into the structure of stochastic matrices having extreme eigenvalues. 

\end{abstract}

\maketitle 

\section{Introduction} 

A square matrix with nonnegative entries is {\emph{stochastic} if each of its row sums is equal to $1$. These matrices are a central object of study because of their fundamental importance to Markov chains, a widely applied class of stochastic processes. For instance, Markov chains have emerged as key tools for understanding phenomena arising in various domains, including web search \cite{PBMW}, economics \cite{MI}, and molecular conformation dynamics \cite{DHFS}. 
The eigenvalues of a stochastic matrix are critical in determining the convergence (or lack thereof) of the corresponding Markov chain, and consequently there is long--standing interest in localising the eigenvalues of stochastic matrices. 

 A classic result of Karpelevi\v c \cite{Kaenglish} describes $\Theta_n,$ the region in the complex plane consisting of all eigenvalues of all $n \times n$ stochastic matrices. That is, 
$$\Theta_n = \{ \lambda | \lambda \mbox{\rm{ is an eigenvalue of a stochastic matrix of order }} n\}.$$ 
The set $\Theta_n$ is readily seen to be star--shaped with respect to the origin, and consequently  Karpelevi\v c's description of $\Theta_n$  focuses on describing its boundary, $\partial \Theta_n$. That description is given in terms of certain one--parameter families of  polynomials; those polynomials, which are categorised into Types 0--III in \cite{JP}, are explored further in Section \ref{sec:2} below. We note in passing that a reparameterisation and sharpening of Karpelevi\v c's original  result can be found in  \cite{KLS}.  

 While the boundary of $\Theta_n$ is well--understood, little is known about the structure of $n \times n$ stochastic matrices having an eigenvalue on $\partial \Theta_n$. An early result of Dmitriev and Dynkin \cite{DD} provides some general information on such matrices, while Johnson and Paparella 
 \cite{JP} construct, for each $\lambda \in \partial \Theta_n,$   a stochastic matrix of order $n$ having $\lambda$ as an eigenvalue.   That paper also poses the question as to what extent the stochastic matrices with an eigenvalue on $\partial \Theta_n$ are unique.

In this paper we address the problem of explicitly describing the stochastic matrices furnishing an eigenvalue on $\partial \Theta_n$. Such matrices can be thought of as extreme in some sense, and as shown below,  
 their structure exhibits a deep connection between their combinatorial and algebraic properties.  In particular, we derive explicit descriptions of the stochastic matrices whose characteristic polynomials are of Type 0 or Type I (Sections \ref{sec:type0} and \ref{sec:type1}, respectively). We then go on to describe the sparsest stochastic matrices whose characteristic polynomials are of Type II or Type III (Sections \ref{sec:type2},  \ref{sec:type3}). We also formulate a conjecture on the Type III case. Taken together, our results sharpen the general structural result of \cite{DD}, and greatly extend the families of stochastic matrices constructed in \cite{JP}. Our results show that, save for the Type 0 case, the stochastic matrices having an eigenvalue on $\partial \Theta_n$ are highly non--unique. 
 
The structure of the paper is as follows. The next three subsections provide the necessary background to and preliminaries for the problem under consideration. Section \ref{sec:2} introduces and analyses the polynomials of Type 0--III, while as noted above, sections \ref{sec:type0}--\ref{sec:type3} contain our main results on Types 0 through III, respectively. Throughout we employ a mix of combinatorial and algebraic techniques, and  include examples to   illustrate the constructions.

\section{Background and Notation}

In order to state Karpelevi\v c's description of $\Theta_n$, we need the following definitions.

\begin{definition}
Given $n \in \mathbb{N}$, the set $$\F_n=\{\sfrac{p}{q}; 0 \leq p < q \leq n, \gcd(p,q)=1\}$$ is called \emph{the set of Farey fractions of order $n$}. 
\end{definition}

\begin{definition}
The pair $(\sfrac{p}{q},\sfrac{r}{s})$ is called \emph{a Farey pair (of order $n$)}, if $\sfrac{p}{q},\sfrac{r}{s} \in \mc F_n$, $\sfrac{p}{q} < \sfrac{r}{s}$, 
and $\sfrac{p}{q}< x <\sfrac{r}{s}$ implies $x \not \in \mc F_n$. Associated with each Farey pair is a parameter $d$, defined by $d\equiv \left \lfloor \frac{n}{q}\right \rfloor$. 

The Farey fractions $\sfrac{p}{q}$ and $\sfrac{r}{s}$  are called \emph{Farey neighbours}, if one of  $(\sfrac{p}{q},\sfrac{r}{s})$ and $(\sfrac{r}{s}, \sfrac{p}{q})$ is a Farey pair.  
\end{definition}

Here is the original description of $\Theta_n.$ 

\begin{theorem}(\cite{Kaenglish}, \cite{I})\label{thm:Karpelevic}
The region $\Theta_n$ is symmetric with respect to the real axis, is included in the unit disc $\{z \in \C \mid |z|\leq 1\},$ and intersects the unit circle  $\{z \in \C \mid |z|=1\}$ at the points $\{e^{ \frac{2 \pi \ii p}{q}} \mid \sfrac{p}{q} \in \F_n\}$.  The boundary of $\Theta_n$ consists of these points and of curvilinear arcs connecting them in circular order. 

Let the endpoints of an arc be $e^{2 \pi \ii p \over q}$ and $e^{2 \pi \ii r \over s}$ with $q<s$. Each of these arcs is given by the following parametric equation: 
 \begin{equation}\label{eq:Karpelevic Poly}
 t^s (t^q-\beta)^{\lfloor {n \over q} \rfloor}=\alpha^{\lfloor {n \over q} \rfloor}t^{q\lfloor {n \over q}\rfloor}, \, \alpha \in [0,1], \, \beta \equiv 1-\alpha.
 \end{equation}
\end{theorem}

\begin{remark} As noted above, we are interested in describing the $n \times n$ stochastic  matrices that realise an eigenvalue  $\lambda \in \partial \Theta_n.$ For a Farey pair of order $n$ $(\sfrac{p}{q},\sfrac{r}{s})$ with $q < s,$ it is possible that $s<n.$ When that is the case then in fact $\lambda \in \Theta_{n-1}$. For this reason we  restrict our attention to the case $s=n$ henceforth, as that is the setting of greatest mathematical interest. 
\end{remark}

\begin{remark}
	Suppose that we are given  a Farey pair of order $n$ $(\sfrac{p}{q},\sfrac{r}{s})$ with $q < s=n.$ Inspecting \eqref{eq:Karpelevic Poly} of  Theorem \ref{thm:Karpelevic}, we see that the corresponding polynomial has one or more roots equal to zero, and since $0 \notin \partial \Theta_n$ when $n\ge 3,$ those zero roots are extraneous to our discussion. 	
	In  simplifying the polynomials by removing those extraneous zero roots, the following cases present themselves; here we follow the terminology introduced in \cite{JP}.
	\begin{itemize}
		\item $q=1$: This leads to the Type 0 polynomials described in Section \ref{sec:type0}; 
		\item $\lfloor {n \over q} \rfloor =1$: This leads to the Type I polynomials described in Section \ref{sec:type1};
		\item $\lfloor {n \over q} \rfloor \ge 2, n<q\lfloor {n \over q} \rfloor$: This leads to the Type II polynomials described in Section \ref{sec:type2} ;  
		\item $\lfloor {n \over q} \rfloor \ge 2, n > q\lfloor {n \over q} \rfloor$: This leads to the Type III polynomials described in Section \ref{sec:type3}. 	\end{itemize}
	Collectively, these polynomials are known as the \emph{reduced Ito polynomials}~\cite{I}. 
	\end{remark}

In our study of stochastic realisations of reduced Ito polynomials we will look at the corresponding weighted directed graphs. Here we review basic definitions and background results that we will need in our analysis. 

\begin{definition}
\emph{A weighted digraph} $\Gamma=(V(\Gamma) , E(\Gamma), w)$ is defined by its set of vertices  $V(\Gamma)=\{1,\ldots,n\}$, set of edges $E(\Gamma) \subset V(\Gamma) \times V(\Gamma)$, and a positive real map $w: E(\Gamma)\rightarrow \R_+$ that defines the weights on $E(\Gamma)$. 
\end{definition}

\begin{definition}
{\emph The adjacency matrix of a weighted digraph} $\Gamma$ with $|V(\Gamma)|=n$ is the $n \times n$ matrix
$A =(a_{ij})$, where where $a_{ij} = w(i,j)$ if $(i, j )\in E(\Gamma)$, and $a_{ij} = 0$ otherwise.
\end{definition}

We follow standard definitions associated with digraphs for notions such as a walk, a path and a cycle. We call a collection of vertex disjoint cycles \emph{a linear digraph}. 
The following theorem explains how the coefficients of the characteristic polynomial of a matrix can be understood from the corresponding digraph. It follows immediately from the Coates formula for the determinant of a square matrix \cite{C}. 

\begin{theorem}\label{thm:coefficients}
Let $\Gamma=(V(\Gamma) , E(\Gamma), w)$ be a weighted digraph on $n$ vertices, $A$ its adjacency
matrix, and $$\det(x I_n -A) = x^n + k_1x^{n-1} + k_2 x^{n-2} +\ldots + k_n.$$
Let $\mc L_i$ denote the set of all linear digraphs $\mc L$ in $\Gamma$, with $|V(\mc L)|=i$. Let $c(\mc L)$ denote the number of cycles in $\mc L$, and $\pi(\mc L):=\prod_{(i,j) \in E(\mc L)} w(i,j)$. 
Then, for each $i \in \{1,\ldots,n\}$, we have:
\begin{equation}\label{eq:ki}
k_i =\sum_{\mc L \in \mc L_i} (-1)^{c(\mc L)}\pi(\mc L).
\end{equation}
\end{theorem} 

In addition to notation already established in this section, we will denote the  $n \times n$ identity matrix by $I_n$ and the $n \times n$ permutation matrix with $1$s in positions  $(i, (i+1)_n), i=1, \ldots, n$, by $C_n$. Furthermore, for $k, n \in \mathbb{N}$ we will use the notation $(k)_n:=k \mod n$, and for  $i,j \in \{1,\ldots,n\}$,  $d_n(i,j):=\min\{(i-j)_n,(j-i)_n\}$ for \emph{the distance between $i$ and $j$ modulo $n$}. Finally, we adopt the convention that $\alpha, \beta \in [0,1], $ with $\alpha+\beta =1$.

\section{Types of Reduced Ito Polynomials}\label{sec:2} 

For any reduced Ito polynomial $f_{\alpha}(t)$ there exists a stochastic matrix $A$ with the characteristic polynomial $f_{\alpha}(t)$, \cite{JP}. Given a reduced Ito polynomial  $f_{\alpha}(t)$ of degree $n$ corresponding to a Farey pair in $\mc F_n$, we aim to better understand the family of stochastic matrices with the characteristic polynomial $f_{\alpha}(t)$.   
Such polynomials occur in one of the following situations
\begin{itemize}
\item Type 0: $ f_{\alpha}(t)=(t-\beta)^n-\alpha^n $, $n=s$, $d=n$, $q=1$,
\item Type I: $f_{\alpha}(t)=t^n-\beta t^{n-q}-\alpha$, $n=s$, $d=1$, $q>\frac{n}{2}$. 
\item Type II: $f_{\alpha}(t)=(t^q-\beta)^{d}-\alpha^{d}t^z$, $n=q d$, $d \geq 2$, $s=qd-z$, $z \in \{1,\ldots, q-1\}$.  
\item Type III: $f_{\alpha}(t)=t^{y}(t^q-\beta)^{d}-\alpha^d$, $n=s=q d+y$,  $y \in \{1,\ldots, q-1\}$, $d \geq 2$. 
\end{itemize}

Clearly, for any weighted digraph $\Gamma$ corresponding to a stochastic matrix and each vertex $v \in V(\Gamma)$, the sum of weights of all outgoing edges from $v$ is equal to $1$. The following result of Dimitriev and Dynkin \cite{DD} imposes further restrictions on a digraph associated with a stochastic matrix realising any of the reduced Ito polynomials listed above.

\begin{theorem}\cite{DD}\label{thm:dd}
Let $\lambda \in \Theta_n \setminus \Theta_{n-1}$ with $\frac{2 \pi k}{n} \leq \arg{\lambda} \leq \frac{2 \pi (k+1)}{n}$. Any $n \times n$ stochastic matrix with an eigenvalue $\lambda$ is permutationally similar to a matrix $A=(a_{i \, j})$ with $a_{i\, j} \neq 0$ only for $j=1+(i+k-1)_n$ or $j=1+(i+k)_n$.  
\end{theorem}

The polynomials $f_{\alpha}$ have several coefficients equal to zero. This restricts the structure of cycles in the corresponding digraph in the following way. 

\begin{proposition}\label{prop:cycle_structure}
Let  $\lambda \in \Theta_n \setminus \Theta_{n-1}$ be a root of reduced Ito polynomial $f_{\alpha}(t)$ of degree $n$,  and $\Gamma$ a weighted digraph whose adjacency matrix is a stochastic matrix with the characteristic polynomial is $f_{\alpha}(t)$. Then $\Gamma$ contains at least one $s$-cycle, at least one $q$-cycle, and no cycles of length different than $s$ and $k q$, $k=1,\ldots, d$. 
\end{proposition}

\begin{proof}
First we note that $\Gamma$ has no cycles of order less than $q$, since the coefficients of $t^{n-i}$, $i=1,\ldots,q-1$, of $f_{\alpha}$ are all equal to zero.
Furthermore, for $k>n-s$  the coefficient of $t^{k}$ is nonzero only if $k$ is divisible by $q$. Since $n-s \leq q-1$, Theorem \ref{thm:coefficients} implies that lengths of cycles in $\Gamma$ are either divisible by $q$ or equal to $s$.  
\end{proof}

In addition to their zero-nonzero pattern,  we will depend on the values of the coefficients of $f_{\alpha}$, to study their stochastic realisations. In particular, we note the following equality between the coefficients.  

\begin{lemma}\label{lem:ItoEquality1}
Let $f_{\alpha}(t)$ be a reduced Ito polynomial of degree $n$ associated with a Farey pair $(\frac{p}{q},\frac{r}{s}) \in \F_n$, $d=\lfloor \frac{n}{q}\rfloor$, and let $k_j$ denote the coefficient of $t^j$. Then $2 d k_{2 q}=(d-1)k_q^2.$
\end{lemma}

\begin{proof}
 The equality can easily be verified by computing $k_{q}$ and $k_{2q}$ of reduced Ito polynomials. Indeed, we have $k_q=-\beta d$ and $k_{2q}=\frac{1}{2}d(d-1) \beta^2$.
\end{proof}

The following inequality is a straightforward variation of the arithmetic-geometric mean inequality. 

\begin{lemma}\label{lem:iq}
Let $\alpha_i \geq 0$, $i=1,\ldots,t$. Then 
\begin{equation}\label{eq:alphas}
2 \sum_{i \neq j} \alpha_i \alpha_j \leq (t-1) \sum_{i=1}^n \alpha_i^2,
\end{equation}
with equality if and only if $\alpha_1=\alpha_2= \ldots=\alpha_t$.
\end{lemma}

\begin{corollary}\label{cor:q_cycles}
Let $\Gamma$ be the digraph of a stochastic matrix whose characteristic polynomial is equal to a reduced Ito polynomial of degree $n$ associated with a Farey pair $(\frac{p}{q},\frac{r}{s}) \in \F_n$, $d=\lfloor \frac{n}{q}\rfloor$. Then $\Gamma$ contains least $d$ $q$-cycles, and if it contains precisely $d$ $q$-cycles, then they are disjoint and of equal weight. 
\end{corollary}

\begin{proof} The conclusion is immediate if $d=1,$ so henceforth we assume that $d \ge 2.$ 
 Let $w_1, w_2, \ldots, w_t$ be the weights of $q$-cycles in $\Gamma$. Since, by Proposition \ref{prop:cycle_structure}, $\Gamma$ contains no cycles of order less than $q$, Theorem \ref{thm:coefficients} applied the coefficient of $t^{n-q}$ of $f_{\alpha}$ gives us  $k_q=-\sum_{i=1}^t w_i$, and hence 
 \begin{equation}\label{eq:square}
 k_q^2=\sum_{i=1}^t w_i^2+ 2\sum_{i \neq j}w_iw_j.
 \end{equation} 
 The same theorem applied to the coefficient of $t^{n-2q}$ implies 
 \begin{equation}\label{eq:ineq}
 \sum_{i \neq j}w_iw_j \geq k_{2q},
 \end{equation} with equality if and only if all $q$-cycles are vertex disjoint and there are no $2q$-cycles in $\Gamma$. Using Lemmas \ref{lem:ItoEquality1} and \ref{lem:iq}, we get
 \begin{align*}
 (d-1)k_q^2&=2d k_{2q}\leq 2d  \sum_{i \neq j}w_iw_j=\frac{2d}{t} \sum_{i \neq j}w_iw_j+(2d-\frac{2d}{t}) \sum_{i \neq j}w_iw_j \\
 &\leq \frac{d(t-1)}{t}\sum_{i=1}^t w_i^2+\frac{2d(t-1)}{t} \sum_{i \neq j}w_iw_j =\frac{d(t-1)}{t}k_q^2.
 \end{align*}
 From here, we conclude that $t \geq d$. Moreover, if $t=d$, then:
 \begin{align*}
 \sum_{i \neq j}w_iw_j =k_{2q}, \text{ and } \\
 \sum_{i \neq j}w_iw_j =(t-1)\sum_{i=1}^d w_i^2.
 \end{align*}
 The first equality forces all $q$-cycles in $\Gamma$ to be disjoint, and  the second one implies $w_1=w_2=\ldots=w_d$ by Lemma \ref{lem:iq}. 
\end{proof}

Now that we have noted some common features of polynomials $f_{\alpha}$ we take a closer look at stochastic realisations of each type separately.

\section{Type 0}\label{sec:type0} 
Realisations of Type 0 polynomials are the easiest to resolve. It is possible to deduce the following from results in \cite{DD}; here we provide an independent proof relying on the work in Section \ref{sec:2}. 

\begin{theorem}
Let $A$ be an $n \times n$ stochastic matrix, with the characteristic polynomial $f_{\alpha}(t)=(t-\beta)^n-\alpha^n$, $\alpha \in (0,1)$. Then $A$ is permutationally  similar to $\beta I_n +\alpha C_n$. 
\end{theorem}

\begin{proof}
Since $f_{\alpha}$ has a root $\lambda \in \partial\Theta_n$ with $0< \arg(\lambda) < \frac{2\pi}{n}$, we get $k=0$ in Proposition \ref{prop:cycle_structure}. This implies that any $n \times n$ stochastic matrix $A$ with characteristic polynomial $f_{\alpha}$ is, up to permutation similarity, of the form $A=D+(I_n-D)C_n$, where $D \in \R^{n \times n}$ is a diagonal matrix with diagonal elements $\alpha_1, \ldots, \alpha_n$, $\alpha_i \in [0,1]$. Since $f_{\alpha}$ is of Type 0 we have  $d=n$ and $q=1$, hence the digraph corresponding to $A$ has to have all its diagonal elements equal to $\beta$, by Corollary \ref{cor:q_cycles}, or equivalently $D=\beta I_n$. 
\end{proof}

Note that a Type 0 Ito polynomial of degree $n$ has a unique stochastic realisation up to permutation similarity.

\section{Type I}\label{sec:type1}

For a stochastic matrix whose characteristic polynomial is Type I (with parameters $n,q,$ say) we find from Proposition \ref{prop:cycle_structure} that the corresponding directed graph can only have cycles of lengths $q$ and $n$. Our next result 
 characterises such directed graphs. 

\begin{lemma}\label{thm:cycles}
Let $q,n \in \mathbb{N},$ satisfy $\frac{n+1}{2}\le q \le n-1$, $\gcd(q,n)=1$. Let  
$\Gamma'$ be a digraph on $n$ vertices containing only cycles of order $q$ and $n$. Then $\Gamma'$ is isomorphic to one of the digraphs $\Gamma$ satisfying $V(\Gamma)=\{1,2,\ldots,n\}$ and $$\mc E_n \cup \{(1,1+(1-q)_n)\} \subseteq E(\Gamma) \subseteq \mc E_n \cup \mc E_{q},$$ where
 $\mc E_n:=\{(j,1+(j)_n); j=1, \ldots, n  \}$ and $\mc E_q:=\{(j, 1+(j-q)_n); j=1, \ldots, n+1-q\}$. Conversely, all such digraphs $\Gamma$ have only cycles of order $q$ and $n$. 
\end{lemma} 
\begin{proof}
First we show that a digraph $\hat \Gamma$ on $n$ vertices with $E(\hat \Gamma )=\mc E_n \cup \mc E_q$ has only cycles of orders $n$ and $q$. 
Note that for  $j=1, \ldots, n+1-q$, we have $1+( j-q)_n   \in \{n+2-q, \ldots, n\}$, unless $2 q=n+1$, $j=n+1-q$, in which case $1+( j-q)_n=1$.

To see the claim, suppose that a cycle $C$ in $\hat \Gamma$ includes the edge $(j_0, 1+(j_0-q)_n)$ with $1+(j_0-q)_n  \in \{n+2-q, \ldots, n\}.$  Since the vertices in $\{n+2-q, \ldots, n\}$ all have out-degree $1$, we find that the edge $(j_0, 1+(j_0-q)_n)$   in $C$ is necessarily followed by the path  $j_0+1-q \rightarrow j_0+2-q \rightarrow \ldots \rightarrow n \rightarrow 1.$ If $C$ includes another edge  $(j_1,1+( j_1-q)_n) \in \mc E_q$ then either  $( j_1+1-q)_n=1$ or $C$ passes through $1$ a second time by the above argument, which is impossible. 
It follows that any cycle in $\hat \Gamma$ contains just one edge from $\mc E_q$, or no such edges; in the former case the cycle has order $q$, while in the latter case the cycle has order $n$.

Let $\Gamma$ be a directed graph containing only cycles of order $n$ and $q$. Since $\Gamma$ contains an $n$-cycle, we can, without loss of generality assume that $\mc E_n \subseteq E(\Gamma).$ Now, any edge not of the form $(j, 1+(j-q)_n)$ would produce a cycle of order different than $q$ and $n$ in $\Gamma$, so all edges in $E(\Gamma)$, not in $\mc E_n$, are of the form $(j,1+(j-q)_n)$.

For $i,j \in \{1,2,\ldots,n\}$, $i \neq j$, define:
\begin{align*}
\mc V_{i,j}&:=\{i,1+(i)_n,1+(i+1)_n,\ldots,1+(j-2)_n,j\} \\
\mc V_{j,i}&:=\{j,1+(j)_n,1+(j+1)_n,\ldots,1+(i-2)_n,i\},
\end{align*} so that $\{1,2,\ldots,n\}=\mc V_{i,j} \cup \mc V_{j,i}$ and $\mc V_{i,j}  \cap \mc V_{j,i}=\{i,j\}$.
Assume 
\begin{equation}\label{eq:ij}
(i,1+(i-q)_n), (j,1+(j-q)_n)\in E(\Gamma),
\end{equation}
 with $|\mc V_{i,j}| \geq n-q+2$ and  $|\mc V_{j,i}|\geq n-q+2$. Then $1+(i-q)_n \in \mc V_{i,j}$, $1+(j-q)_n  \in \mc V_{j,i}$, and 
$$i \rightarrow (i+1-q)_n \rightarrow (i+2-q)_n \rightarrow \ldots \rightarrow j \rightarrow  (j+1-q)_n \rightarrow (j+2-q)_n \rightarrow \ldots \rightarrow i $$ is a cycle in $\Gamma$ of length $2q-n$, a contradiction. 
(If $1+(i-q)_n=j$ then the corresponding cycle is:
$i \rightarrow 1 +(i-q)_n= j \rightarrow 1+ (j-q)_n \rightarrow 1+(j+1-q)_n \rightarrow \ldots \rightarrow i.$)
Hence, for any two vertices $i,j$ satisfying \eqref{eq:ij} we have $|\mc V_{i,j}| \leq n-q+1$ or  $|\mc V_{j,i}|\leq n-q+1$. Among those, let $i_0$ and $j_0$ be chosen so that $\min\{|\mc V_{i_0,j_0}| ,|\mc V_{j_0,i_0}| \}$ is maximal. Without loss of generality we may assume $\min\{|\mc V_{i_0,j_0}| ,|\mc V_{j_0,i_0}| \}=|\mc V_{i_0,j_0}|\leq n-q+1$ and $i_0=1$; hence $j_0 \leq n+1-q$, as desired. 
\end{proof}

Now that we understand the restrictions on the edges of the associated digraph, we can characterise all stochastic realisations of Type I polynomials. 

\begin{theorem}\label{thm:TypeIR}
Let $n$ and $q$ be positive integers satisfying $2 \leq q <n$, $2 q>n$.  An $n\times n$ stochastic matrix has the characteristic polynomial $f_\alpha(t)=t^n - (1-\alpha)t^{n-q}-\alpha$ if and only if it is permutationally similar to a matrix of the form: 
$$A=(D \oplus I_{q-1})C_n+ ((I_{n+1-q}-D) \oplus 0_{q-1})C_n^{n-q+1},$$ 
where  $D \in \R^{(n+1-q) \times (n+1-q)}$ is a diagonal matrix, with diagonal entries $\alpha_i \in (0,1]$ satisfying $\prod_{i=1}^{n+1-q} \alpha_i =\alpha$.
\end{theorem} 

\begin{proof} Suppose that a stochastic matrix has characteristic polynomial $f_\alpha$. Then the associated digraph has only cycles of lengths $n$ and $q$. Lemma \ref{thm:cycles} then determines the pattern   of the matrix, from which we deduce that our matrix must be permutationally similar to a matrix $A$ of the form $(D \oplus I_{q-1})C_n+ ((I_{n+1-q}-D) \oplus 0_{q-1})C_n^{n-q+1}.$
The weights $\alpha_i$ and $1-\alpha_i$ with $\alpha_i \in (0,1]$ are required for the matrix $A$ to be stochastic, and $\prod_{i=1}^{n+1-q} \alpha_i =\alpha$ fixes the constant term of the characteristic polynomial of $A$ to be equal to $-\alpha$. 

It is straightforward to determine the converse -- i.e. that any matrix $A$ as described above has $f_\alpha$ as its characteristic polynomial. 
\end{proof}

\begin{remark}
The sparsest realisation of Type I polynomial is  unique up to permutation similarity,  and occurs in Theorem \ref{thm:TypeIR} for the choice $\alpha_1=\alpha$, $\alpha_j=1$, $j=2,\ldots, n-q+1$. 
\end{remark}

\begin{example}
Let $f_{\alpha}(t)=t^{5}- 
 t(1-\alpha) - \alpha$.
  According to Theorem \ref{thm:TypeIR} any stochastic matrix with the characteristic polynomial $f_{\alpha}(t)$ is permutationally similar to a matrix of the form
$$A=\left(
\begin{array}{ccccc}
 0 & \alpha_1 & 1-\alpha_1 & 0 & 0 \\
 0 & 0 & \alpha_2 & 1-\alpha_2 & 0 \\
 0 & 0 & 0 & 1 & 0 \\
 0 & 0 & 0 & 0 & 1 \\
 1 & 0 & 0 & 0 & 0 \\
\end{array}
\right),$$
where $\alpha_i \in (0,1]$ and $\alpha_1\alpha_2 =\alpha$. 
\end{example}

\section{Type II}\label{sec:type2} 

Since each Ito polynomial corresponds to a particular   Karpelevi\v c arc and Farey pair, it is useful to express $k$ in Theorem \ref{thm:dd} using the parameters determining that Farey pair. For Type II reduced Ito polynomials, this is done in the lemma below. 

\begin{lemma}\label{cor:DDTII}
Let $n=qd$, $q \geq 2$,  and let $A'$ be an $n \times n$ stochastic matrix with an eigenvalue $\lambda \in \mc K_n(\{q,s\})\setminus \Theta_{n-1}$. Then $A'$ is permutationally similar to a matrix $A=(a_{ij})$ with $a_{i j} \neq 0$ only for $j=1+(i+p d-1)_n$ or $j=1+(i+pd)_n$.  
\end{lemma}

\begin{proof}
By Lemma \ref{thm:dd}, $A'$ is, under permutation similarity, similar to a matrix $A=(a_{ij})$ with  $a_{i\, j} \neq 0$ only for $i=1+(j+k-1)_n$ or $i=1+(j+k)_n$, where $\frac{2 \pi k}{n} \leq \arg{\lambda} \leq \frac{2 \pi (k+1)}{n}$. 
On the other hand, we are assuming  $\frac{2 \pi p}{q} \leq \arg{\lambda} \leq \frac{2 \pi r}{qd}$.
Since $\frac{p}{q}=\frac{pd}{qd}$, and $\frac{r}{s}\leq \frac{pd+1}{qd}$, the claim follows. 
\end{proof}

Lemma \ref{cor:DDTII} restricts the edges that can appear in the digraph $\Gamma$. In particular, we know that any realising matrix $A$ is going to be of the form $DC_n^{pd}+(I_n-D)C_n^{pd+1}$, for some diagonal matrix $D$ whose diagonal entries are contained in $[0,1]$. As some of elements of $D$ can be equal to $0$ or $1$, at this point we do not know which of those edges are actually included in $\Gamma$. In the next lemma we take a closer look at $q$-cycles that can appear in a digraph of a matrix of the form $A=DC_n^{pd}+(I_n-D)C_n^{pd+1}$.

\begin{lemma}\label{lem:n=qd}
Let $\zeta=(\sfrac{p}{q},\sfrac{r}{s}) \in \mc F_n$, $n=qd$, $q\geq 2$. Let $\Gamma$ be a digraph with vertex set $V=\{1,2,\ldots,n\}$ and edge set $$E=\{(i,1+(i+p d-1)_n); i=1,\ldots,n\} \cup \{(i,1+(i+pd)_n); i=1,\ldots,n\}.$$ Then:
\begin{enumerate}
\item Edges of the form $(i,1+(i+pd-1)_n)$, $t=1,\ldots,n$, form $d$ disjoint $q$-cycles in $\Gamma$.
\item None of the edges of the form  $(i,1+(i+pd)_n)$ is contained in a $q$-cycle in $\Gamma$. 
\end{enumerate}
\end{lemma}

\begin{proof}
Note that vertices $i, 1+(i-1+pd)_n, 1+(i-1+2pd)_n, \ldots, 1+(i-1+(q-1)pd)_n$ are contained in a $q$-cycle in $\Gamma$, since $1+(i-1+qpd)_n=i$. This proves the first claim. 

Assume that the second item does not hold. Then we can, without loss of generality, assume that $(1, 1+(1+pd)_n)$ is contained in a $q$-cycle, and let $1 \leq j_1 \leq q$ be the number of edges of type $(i,1+(i+pd)_n)$ in this cycle.  Hence, $(1+q pd+j_1)_n=1$, or equivalently $(q pd+j_1)_{qd}=0$. 
Now we see that $qpd+j_1$ has to be divisible by $q$, so $j_1=q$. From here we get that $pd+1$ has to be divisible by $d$, a contradiction.
\end{proof}

To get further restrictions on the cycles we need to consider conditions on the weights of the cycles coming from the coefficients of the reduced Ito polynomial $f_{\alpha}(t)$. Our first observation tells us that all the $q$-cycles identified in Lemma \ref{lem:n=qd} have to appear. 
 
\begin{lemma}\label{prop:cyclesTII}
Let  $f_{\alpha}(t)=(t^q-\beta)^{d}-\alpha^{d}t^z$, $s=qd-z$, $z \in \{0,\ldots, q-1\}$ be a Type II reduced Ito polynomial. The  digraph $\Gamma$ corresponding to any stochastic realisation of $f_{\alpha}(t)$ has $d$ disjoint $q$-cycles of equal weight $\beta$, and no cycles of order $kq$ for $k \geq 2$.
\end{lemma}

\begin{proof}
By Lemma \ref{lem:n=qd} we know that $\Gamma$ has at most $d$ $q$-cycles; hence by Corollary \ref{cor:q_cycles} we find that  $\Gamma$ has precisely $d$ $q$-cycles, all of equal weight $\beta$. 
From here we determine that for $b=1, \ldots, d$, the coefficient of $t^{qb}$ is equal to the contribution of the $q$-cycles, hence the $q$-cycles are the only cycles of lengths divisible by $q$ in $\Gamma$. 
\end{proof}

Let  $f_{\alpha}(x)=(x^q-\beta)^{d}-\alpha^{d}x^z$, $n=q d$, $d \geq 2$, $s=qd-z$, $z \in \{0,\ldots, q-1\}$ be a Type II reduced Ito polynomial, and let $A=(a_{i,j})$ be a stochastic matrix with the characteristic polynomial $f_{\alpha}(t)$. Based on the observations above,
we know that any stochastic matrix $A$  is of the form $A=DC_n^{pd}+(I_n-D)C_n^{pd+1}$, where $D=\mathrm{diag}(\beta_1, \ldots,\beta_n)$ satisfies:
$\prod_{k=0}^{q-1} (\beta_{1+(i-1+k p d)_n})=\beta$ for $i=1,\ldots,d$. The corresponding digraph of any matrix of this form will have $d$ disjoint $q$-cycles, and no other cycles of order $k q$. From Proposition \ref{prop:cycle_structure} we know that the only other possible cycles are of order $s=qd-z$. This restriction forces some $\beta_i$ to be equal to $1$. 
To better understand this restriction, we use permutation similarity on $A$, so that vertices in each $q$-cycle are numbered  consecutively, as follows: 
\begin{enumerate}
\item For  $t=0,\ldots,d-1$, let us define: $$\mc E_t:=\{(1+t q+i, 1+t q +(i+1)_q), i=0, \ldots, q-1\}.$$ For a fixed $t$ those elements correspond to a $q$-cycle that we denote by $\Gamma_t$. Hence:
$$\prod_{(i,j) \in \mc E_t} a_{i,j} =1-\alpha.$$
\item Let $\mc E_{t,(t+1)_d}$ be the set of all edges in $\Gamma$ between $\Gamma_t$  and $\Gamma_{(t+1)_d}$. We know that $\mc E_{t,(t+1)_d}$ is not empty, and up to permutation similarity we can assume:
$$\mc E_{t, t+1} \subseteq \{(1+t q+i, 1+(t+1) q+i); i= 0,\ldots,q-1\},$$
for $t=0,\ldots,d-2$, and 
$$ \mc E_{d-1,0} \subseteq \{(1+(d-1) q+i, 1+(z+d+i)_q); i =0,\ldots,q-1\}.$$
\end{enumerate}
To recap, the set of edges in $\Gamma$ is equal to $(\cup_{t=0}^{d-1}  \mc E_t)\cup(\cup_{t=0}^{d-1}\mc E_{t,(t+1)_d})$, where the $\mc E_t$ are already determined, and we only have partial information for the  $\mc E_{t,(t+1)_d}$. 
The observation that the only cycles in $\Gamma$ that contain edges from $\mc E_{t,(t+1)_d}$, $t=0,\ldots,d-1$, are of order $n-z$, poses restrictions on those sets. Before we look at those in detail, we can write down all stochastic realisations of $f_{\alpha}(t)$ with minimal number of edges, i.e. realisations that contain a single $(n-z)$-cycle. 

Any collection of edges $\mc E=\{e_t=(a_t,b_t) \in \mc E_{t,(t+1)_d}$, $t=0,\ldots,d-1\}$, is contained on a cycle $\Gamma(\mc E)$ in $\Gamma$, that, apart from edges in $\mc E$, also contains edges from the sets $\mc E_t$, $t=0,\ldots,d-1$. Since we require this cycle to have $n-z$ vertices, we get: 
\begin{equation}\label{eq:n-z}
n-z=\sum_{t=0}^{d-1} ((a_t-b_{(t-1)_q})_q+1).
\end{equation}
 Furthermore, if the permutation similarity is fixed as above, then the $b_t$ are defined by $a_t$ as follows:
\begin{align*}
a_t&=1+t q+i_t, \, &b_t&=1+(t+1)q+i_t \text{ for }t=0,\ldots,d-2, \\
a_{d-1}&=1+(d-1) q+i_{d-1}, \, &b_{d-1}&=1+(z+d+i_{d-1})_q.
\end{align*}
Hence \eqref{eq:n-z} becomes:
\begin{equation}\label{eq:length}
n-z-d=(i_0-(z+d+i_{d-1})_q)_q+\sum_{t=1}^{d-1}(i_t-i_{t-1})_q.
\end{equation}
One way to approach the equation above is to write $n-z-d=\sum_{i=0}^{d-1} x_i$, $x_i \in \{0,\ldots,q-1\}$, and demand: 
$$x_0=(i_0-(z+d+i_{d-1})_q)_q  \text{ and } x_t=(i_t-i_{t-1})_q \text{ for } t=1,\ldots,d-1.$$
In other words, we ask for $x_t+1$ vertices of $\Gamma_t$ to be included in $\Gamma(\mc E)$. 
Note, that the choices of $e_0$ and a partition of $n-z-d$ define $i_t$, and hence $\mc E$, uniquely.  
Indeed, without loss of generality we may take $i_0=0$, then 
$i_t=(\sum_{i=1}^t x_i)_q$. This discussion can be summarised as follows.

\begin{proposition}\label{thm:TII}
Let $f_{\alpha}(t)=(t^q-\beta)^{d}-\alpha^{d}t^z$ be a reduced Ito polynomial of Type II, where  $n=q d$, $d \geq 2$ and $z \in \{0,\ldots, q-1\}$.

Then $A$ is a stochastic matrix with the characteristic polynomial $f_{\alpha}(t)$, and the minimal possible number of nonzero elements, if and only if   
 there exists a partition $n-z-d=\sum_{i=0}^{d-1} x_i$, $x_i \in \{0,\ldots,q-1\}$, so that
$A$ is, up to permutation similarity, equal to a matrix of the form: 
$$A=D\left(\oplus_{k=1}^d  C_q\right)+\alpha \sum_{t=0}^{d-1} E_{a_t,b_t},$$
where
\begin{align*}
a_0&=1, \, a_t=1+t q+ (\sum_{i=1}^t x_i)_q, \, t=1, \ldots,d-1  \\
b_0&=1+q, \, b_t=1+(t+1)q+ (\sum_{i=1}^t x_i)_q, \, t=1, \ldots,d-2, b_{d-1}=1+ (q-x_0)_q,  
\end{align*}
and $D$ is a diagonal matrix with $a_t$-th diagonal element equal to $1-\alpha$, $t=0,\ldots, d-1$, and all other elements equal to $1$. 
\end{proposition}

\begin{example}\label{ex:TIISingle}
Let us consider all the sparsest stochastic realisations of the polynomial $f_{\alpha}(t)=\left(x^4-(1-\alpha) \right)^3-\alpha ^3 x^3$ that are not equivalent under permutation similarity. Since $q=4$, $d=3$, and $z=3$, we need $x_i \in \{0,1,2,3\}$ with $x_0+x_1+x_2=6$. We have four relevant partitions of $6$, namely $6=0+3+3$,  $6=2+2+2$, $6=1+2+3$, and $6=1+3+2$. (Note that $6=1+2+3$ and $6=2+3+1$, say, result in permutationally equivalent realisations.) Below we list the realisations, together with their associated directed graphs. The choice $x_0=0$, $x_1=x_2=3$ results in: 
$$A_1=
\left(
\begin{array}{cccccccccccc}
 0 & 1-\alpha  & 0 & 0 & \alpha  & 0 & 0 & 0 & 0 & 0 & 0 & 0 \\
 0 & 0 & 1 & 0 & 0 & 0 & 0 & 0 & 0 & 0 & 0 & 0 \\
 0 & 0 & 0 & 1 & 0 & 0 & 0 & 0 & 0 & 0 & 0 & 0 \\
 1 & 0 & 0 & 0 & 0 & 0 & 0 & 0 & 0 & 0 & 0 & 0 \\
 0 & 0 & 0 & 0 & 0 & 1 & 0 & 0 & 0 & 0 & 0 & 0 \\
 0 & 0 & 0 & 0 & 0 & 0 & 1 & 0 & 0 & 0 & 0 & 0 \\
 0 & 0 & 0 & 0 & 0 & 0 & 0 & 1 & 0 & 0 & 0 & 0 \\
 0 & 0 & 0 & 0 & 1-\alpha  & 0 & 0 & 0 & 0 & 0 & 0 & \alpha  \\
 0 & 0 & 0 & 0 & 0 & 0 & 0 & 0 & 0 & 1 & 0 & 0 \\
 0 & 0 & 0 & 0 & 0 & 0 & 0 & 0 & 0 & 0 & 1 & 0 \\
 \alpha  & 0 & 0 & 0 & 0 & 0 & 0 & 0 & 0 & 0 & 0 & 1-\alpha  \\
 0 & 0 & 0 & 0 & 0 & 0 & 0 & 0 & 1 & 0 & 0 & 0 \\
\end{array}
\right)$$
\begin{center}
\includegraphics[height=5.5cm,keepaspectratio]{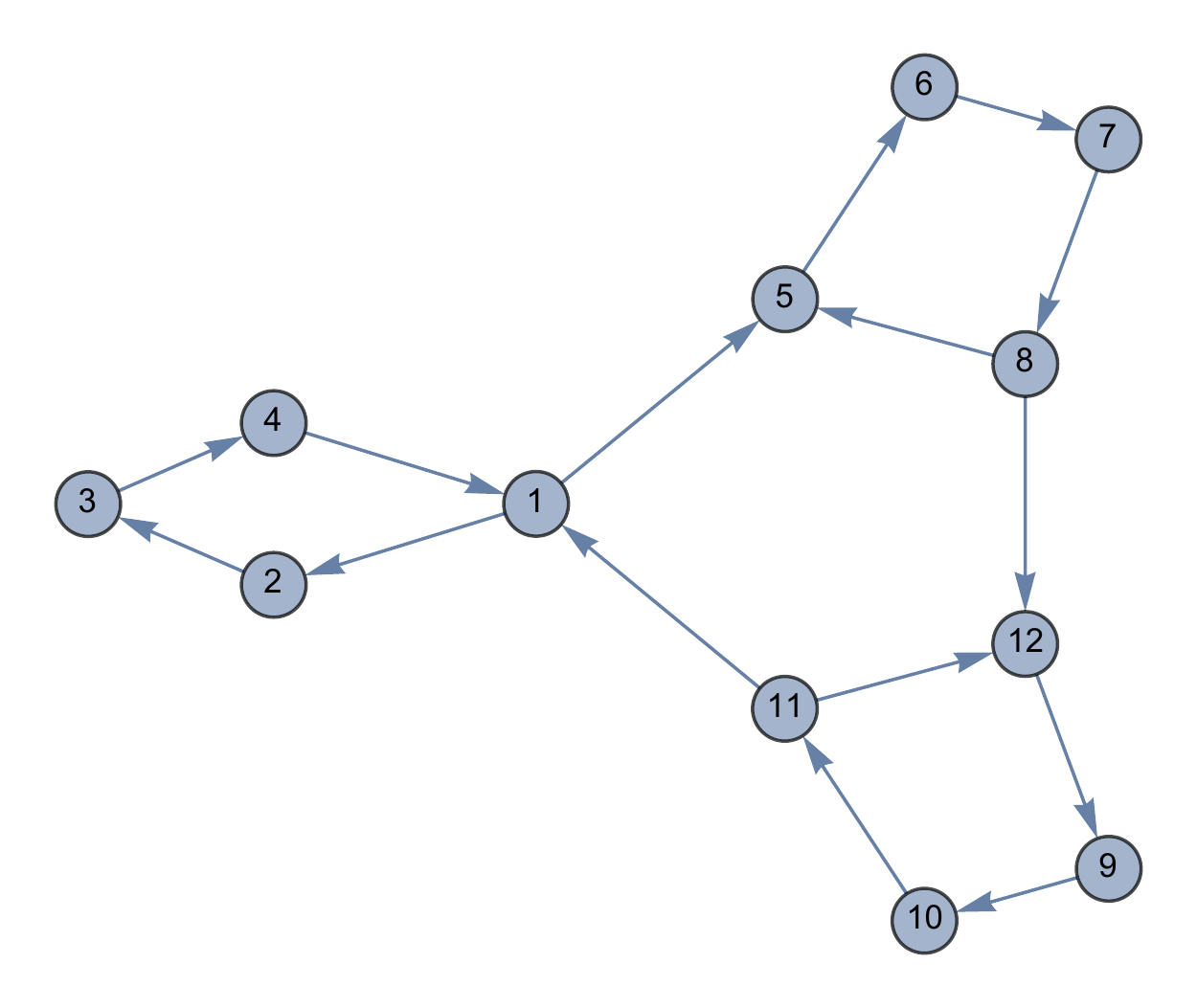}
\end{center}

 The second partition $x_0=x_1=x_2=2$ gives: 
 $$A_2=\left(
\begin{array}{cccccccccccc}
 0 & 1-\alpha  & 0 & 0 & \alpha  & 0 & 0 & 0 & 0 & 0 & 0 & 0 \\
 0 & 0 & 1 & 0 & 0 & 0 & 0 & 0 & 0 & 0 & 0 & 0 \\
 0 & 0 & 0 & 1 & 0 & 0 & 0 & 0 & 0 & 0 & 0 & 0 \\
 1 & 0 & 0 & 0 & 0 & 0 & 0 & 0 & 0 & 0 & 0 & 0 \\
 0 & 0 & 0 & 0 & 0 & 1 & 0 & 0 & 0 & 0 & 0 & 0 \\
 0 & 0 & 0 & 0 & 0 & 0 & 1 & 0 & 0 & 0 & 0 & 0 \\
 0 & 0 & 0 & 0 & 0 & 0 & 0 & 1-\alpha  & 0 & 0 & \alpha  & 0 \\
 0 & 0 & 0 & 0 & 1 & 0 & 0 & 0 & 0 & 0 & 0 & 0 \\
 0 & 0 & \alpha  & 0 & 0 & 0 & 0 & 0 & 0 & 1-\alpha  & 0 & 0 \\
 0 & 0 & 0 & 0 & 0 & 0 & 0 & 0 & 0 & 0 & 1 & 0 \\
 0 & 0 & 0 & 0 & 0 & 0 & 0 & 0 & 0 & 0 & 0 & 1 \\
 0 & 0 & 0 & 0 & 0 & 0 & 0 & 0 & 1 & 0 & 0 & 0 \\
\end{array}
\right),$$
with corresponding directed graph:
\begin{center}
\includegraphics[height=5.5cm,keepaspectratio]{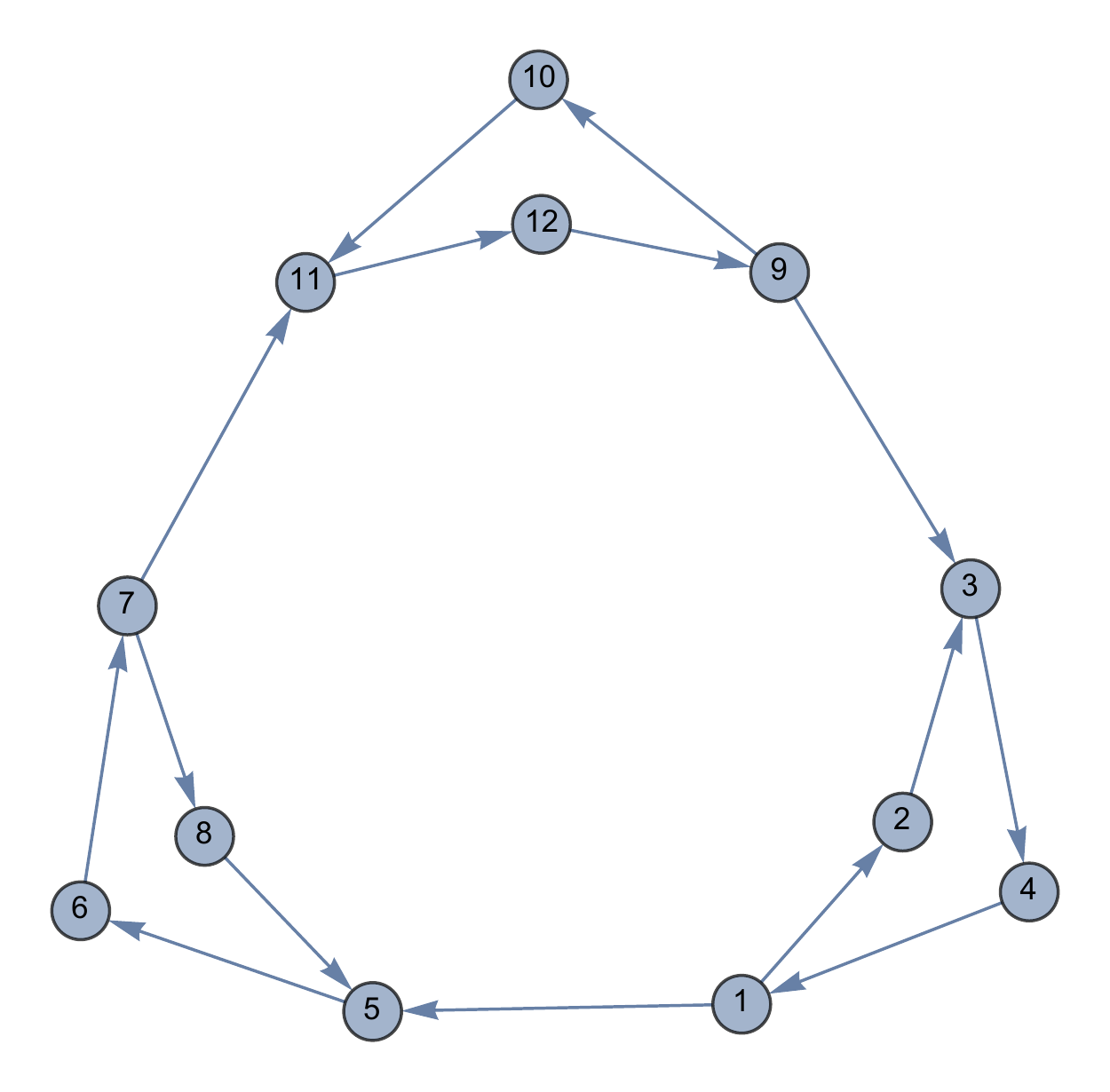}
\end{center}

The choice $x_0=1$, $x_1=2$, $x_2=3$, gives: 
$$A_3=\left(
\begin{array}{cccccccccccc}
 0 & 1-\alpha  & 0 & 0 & \alpha  & 0 & 0 & 0 & 0 & 0 & 0 & 0 \\
 0 & 0 & 1 & 0 & 0 & 0 & 0 & 0 & 0 & 0 & 0 & 0 \\
 0 & 0 & 0 & 1 & 0 & 0 & 0 & 0 & 0 & 0 & 0 & 0 \\
 1 & 0 & 0 & 0 & 0 & 0 & 0 & 0 & 0 & 0 & 0 & 0 \\
 0 & 0 & 0 & 0 & 0 & 1 & 0 & 0 & 0 & 0 & 0 & 0 \\
 0 & 0 & 0 & 0 & 0 & 0 & 1 & 0 & 0 & 0 & 0 & 0 \\
 0 & 0 & 0 & 0 & 0 & 0 & 0 & 1-\alpha  & 0 & 0 & \alpha  & 0 \\
 0 & 0 & 0 & 0 & 1 & 0 & 0 & 0 & 0 & 0 & 0 & 0 \\
 0 & 0 & 0 & 0 & 0 & 0 & 0 & 0 & 0 & 1 & 0 & 0 \\
 0 & 0 & 0 & \alpha  & 0 & 0 & 0 & 0 & 0 & 0 & 1-\alpha  & 0 \\
 0 & 0 & 0 & 0 & 0 & 0 & 0 & 0 & 0 & 0 & 0 & 1 \\
 0 & 0 & 0 & 0 & 0 & 0 & 0 & 0 & 1 & 0 & 0 & 0 \\
\end{array}
\right)$$
\begin{center}
\includegraphics[height=5.5cm,keepaspectratio]{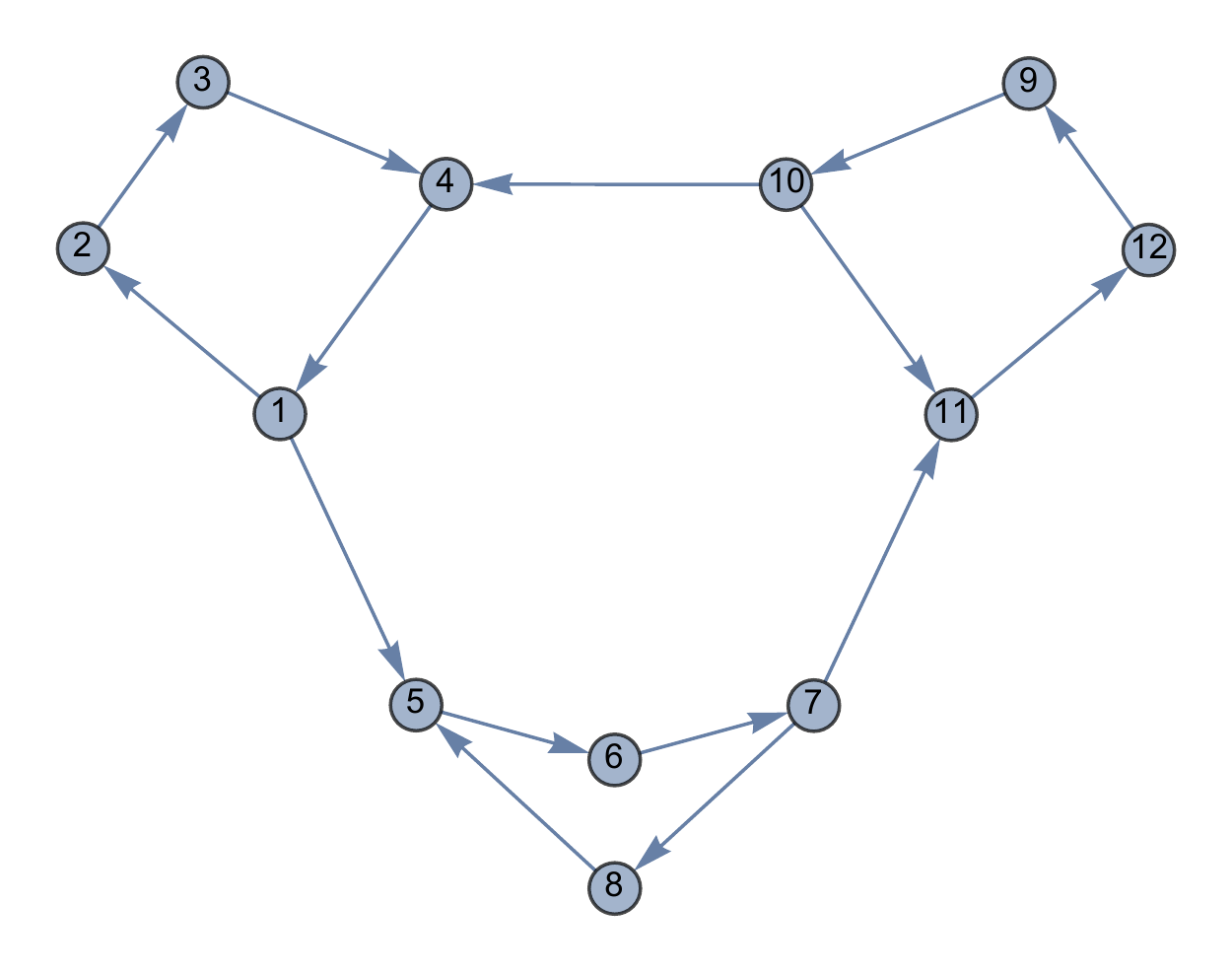}
\end{center}
Finally,  $x_0=1$, $x_1=3$, $x_2=2$, results in:
$$A_4=\left(
\begin{array}{cccccccccccc}
 0 & 1-\alpha  & 0 & 0 & \alpha  & 0 & 0 & 0 & 0 & 0 & 0 & 0 \\
 0 & 0 & 1 & 0 & 0 & 0 & 0 & 0 & 0 & 0 & 0 & 0 \\
 0 & 0 & 0 & 1 & 0 & 0 & 0 & 0 & 0 & 0 & 0 & 0 \\
 1 & 0 & 0 & 0 & 0 & 0 & 0 & 0 & 0 & 0 & 0 & 0 \\
 0 & 0 & 0 & 0 & 0 & 1 & 0 & 0 & 0 & 0 & 0 & 0 \\
 0 & 0 & 0 & 0 & 0 & 0 & 1 & 0 & 0 & 0 & 0 & 0 \\
 0 & 0 & 0 & 0 & 0 & 0 & 0 & 1 & 0 & 0 & 0 & 0 \\
 0 & 0 & 0 & 0 & 1-\alpha  & 0 & 0 & 0 & 0 & 0 & 0 & \alpha  \\
 0 & 0 & 0 & 0 & 0 & 0 & 0 & 0 & 0 & 1 & 0 & 0 \\
 0 & 0 & 0 & \alpha  & 0 & 0 & 0 & 0 & 0 & 0 & 1-\alpha  & 0 \\
 0 & 0 & 0 & 0 & 0 & 0 & 0 & 0 & 0 & 0 & 0 & 1 \\
 0 & 0 & 0 & 0 & 0 & 0 & 0 & 0 & 1 & 0 & 0 & 0 \\
\end{array}
\right)$$
\begin{center}
\includegraphics[height=5.5cm,keepaspectratio]{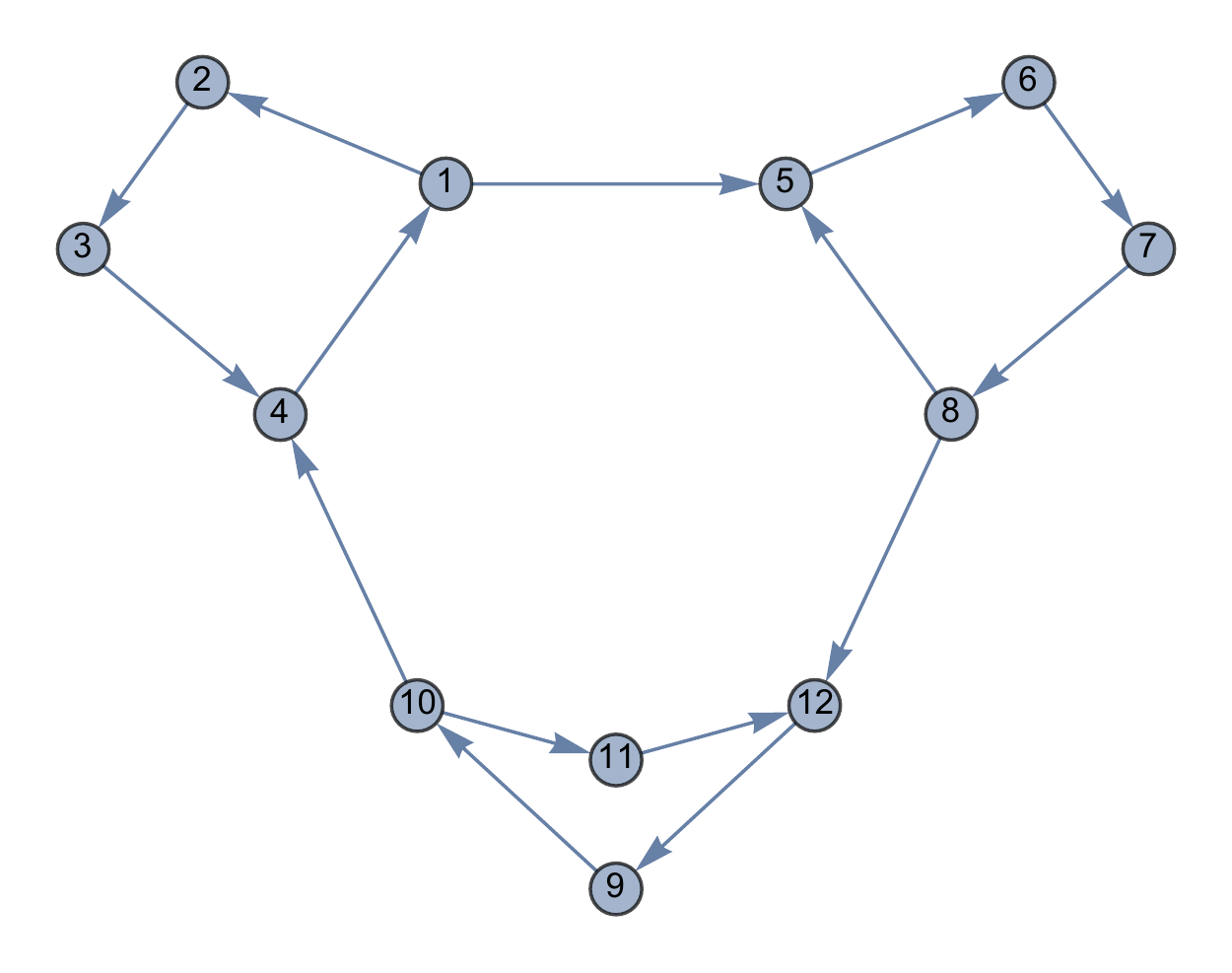}
\end{center}
\end{example}

Now that we understand all stochastic realisations of a Type II Ito polynomial where the associated directed graph $\Gamma$ has a single $(n-z)$-cycle, we need to consider the possibility that there might be more than one $(n-z)$-cycle in $\Gamma$. Indeed this can happen, as there can be be more than one edge between $\Gamma_t$ and $\Gamma_{(t+1)_d}$ in $\Gamma$.

Let us assume that $\Gamma_0$ is a directed graph containing a single $(n-z)$-cycle $\Gamma(\mc E)$ as described above. Let $\Gamma_1$ be obtained from $\Gamma_0$ by adding one edge: $e_{t}'=(a_t',b_t') \in \mc E_{t,(t+1)}$, $t \in \{0,\ldots,d-1\}$. 
The directed graph $\Gamma_1$ formed in this way has an additional cycle $\mc E'$ obtained from $\mc E$ by replacing $e_t$ by $e_{t}'$. We want $\Gamma(\mc E')$ again to be of order $n-z$. Note that in \eqref{eq:n-z} only two terms are affected, and we require: 
\begin{equation}\label{eq:2cycles}
(a_t-b_{t-1})_q+(a_{(t+1)_d}-b_{t})_q=(a_t'-b_{t-1})_q+(a_{(t+1)_d}-b_{t}')_q.
\end{equation}

We can generalise this argument by assuming that $\Gamma_k$ is a graph obtained from $\Gamma_0$ by adding $k$ edges in such a way that $\Gamma_k$ still satisfies the conditions of Lemma \ref{prop:cyclesTII}. As each added edge doubles the number of $(n-z)$-cycles,  $\Gamma_k$ has $2^k$ of them: $\Gamma(\mc E_j)$, $j=1,\ldots,2^k$.
We can add an additional edge $e_{t}'$ to $\Gamma_k$, if the equation analogous to the equation \eqref{eq:2cycles}, can be satisfied for each $\Gamma_k$, $k=1,\ldots,2^k$. 

\begin{example}
Let us investigate this process, by considering which edges can be added to a directed graph associated with $A_1$ in Example \ref{ex:TIISingle}, while preserving the characteristic polynomial of the associated matrix. No edges can be added between $\Gamma_1$ and $\Gamma_2$. 
Three edges can be added between $\Gamma_0$ and $\Gamma_1$, obtaining the realising matrix:
$$A_{11}= \left(
\begin{array}{cccccccccccc}
 0 & \alpha_1 & 0 & 0 & 1-\alpha_1 & 0 & 0 & 0 & 0 & 0 & 0 & 0 \\
 0 & 0 & \alpha_2 & 0 & 0 & 1-\alpha_2 & 0 & 0 & 0 & 0 & 0 & 0 \\
 0 & 0 & 0 & \alpha_3 & 0 & 0 & 1-\alpha_3 & 0 & 0 & 0 & 0 & 0 \\
 \frac{1-\alpha }{\alpha_1 \alpha_2 \alpha_3} & 0 & 0 & 0 & 0 & 0 & 0 &
   1-\frac{1-\alpha }{\alpha_1 \alpha_2 \alpha_3} & 0 & 0 & 0 & 0 \\
 0 & 0 & 0 & 0 & 0 & 1 & 0 & 0 & 0 & 0 & 0 & 0 \\
 0 & 0 & 0 & 0 & 0 & 0 & 1 & 0 & 0 & 0 & 0 & 0 \\
 0 & 0 & 0 & 0 & 0 & 0 & 0 & 1 & 0 & 0 & 0 & 0 \\
 0 & 0 & 0 & 0 & 1-\alpha  & 0 & 0 & 0 & 0 & 0 & 0 & \alpha  \\
 0 & 0 & 0 & 0 & 0 & 0 & 0 & 0 & 0 & 1 & 0 & 0 \\
 0 & 0 & 0 & 0 & 0 & 0 & 0 & 0 & 0 & 0 & 1 & 0 \\
 \alpha  & 0 & 0 & 0 & 0 & 0 & 0 & 0 & 0 & 0 & 0 & 1-\alpha  \\
 0 & 0 & 0 & 0 & 0 & 0 & 0 & 0 & 1 & 0 & 0 & 0 \\
\end{array}
\right).$$
Alternatively, three edges can be added between $\Gamma_2$ and $\Gamma_0$ with the resulting realising matrix: 
 $$A_{12}=\left(
\begin{array}{cccccccccccc}
 0 & 1-\alpha  & 0 & 0 & \alpha  & 0 & 0 & 0 & 0 & 0 & 0 & 0 \\
 0 & 0 & 1 & 0 & 0 & 0 & 0 & 0 & 0 & 0 & 0 & 0 \\
 0 & 0 & 0 & 1 & 0 & 0 & 0 & 0 & 0 & 0 & 0 & 0 \\
 1 & 0 & 0 & 0 & 0 & 0 & 0 & 0 & 0 & 0 & 0 & 0 \\
 0 & 0 & 0 & 0 & 0 & 1 & 0 & 0 & 0 & 0 & 0 & 0 \\
 0 & 0 & 0 & 0 & 0 & 0 & 1 & 0 & 0 & 0 & 0 & 0 \\
 0 & 0 & 0 & 0 & 0 & 0 & 0 & 1 & 0 & 0 & 0 & 0 \\
 0 & 0 & 0 & 0 & 1-\alpha  & 0 & 0 & 0 & 0 & 0 & 0 & \alpha  \\
 0 & 0 & 1-\alpha_1 & 0 & 0 & 0 & 0 & 0 & 0 & \alpha_1 & 0 & 0 \\
 0 & 0 & 0 & 1-\alpha_2 & 0 & 0 & 0 & 0 & 0 & 0 & \alpha_2 & 0 \\
 1-\alpha_3 & 0 & 0 & 0 & 0 & 0 & 0 & 0 & 0 & 0 & 0 & \alpha_3 \\
 0 & 1-\frac{1-\alpha }{\alpha_1 \alpha_2 \alpha_3} & 0 & 0 & 0 & 0 & 0 & 0 &
   \frac{1-\alpha }{\alpha_1 \alpha_2 \alpha_3} & 0 & 0 & 0 \\
\end{array}
\right).$$
It is also possible to add two edges between $\Gamma_0$ and $\Gamma_1$, and a further two edges between $\Gamma_2$ and $\Gamma_0$:
 $$A_{13}=\left(
\begin{array}{cccccccccccc}
 0 & \alpha_1 & 0 & 0 & 1-\alpha_1 & 0 & 0 & 0 & 0 & 0 & 0 & 0 \\
 0 & 0 & \frac{1-\alpha }{\alpha_1} & 0 & 0 & 1-\frac{1-\alpha }{\alpha_1} & 0 & 0 &
   0 & 0 & 0 & 0 \\
 0 & 0 & 0 & 1 & 0 & 0 & 0 & 0 & 0 & 0 & 0 & 0 \\
 1 & 0 & 0 & 0 & 0 & 0 & 0 & 0 & 0 & 0 & 0 & 0 \\
 0 & 0 & 0 & 0 & 0 & 1 & 0 & 0 & 0 & 0 & 0 & 0 \\
 0 & 0 & 0 & 0 & 0 & 0 & 1 & 0 & 0 & 0 & 0 & 0 \\
 0 & 0 & 0 & 0 & 0 & 0 & 0 & 1 & 0 & 0 & 0 & 0 \\
 0 & 0 & 0 & 0 & 1-\alpha  & 0 & 0 & 0 & 0 & 0 & 0 & \alpha  \\
 0 & 0 & 1-\alpha_1' & 0 & 0 & 0 & 0 & 0 & 0 & \alpha_1' & 0 & 0 \\
 0 & 0 & 0 & 1-\alpha_2' & 0 & 0 & 0 & 0 & 0 & 0 & \alpha_2' & 0 \\
 1-\frac{1-\alpha }{\alpha_1' \alpha_2'} & 0 & 0 & 0 & 0 & 0 & 0 & 0 & 0 & 0 & 0 &
   \frac{1-\alpha }{\alpha_1' \alpha_2'} \\
 0 & 0 & 0 & 0 & 0 & 0 & 0 & 0 & 1 & 0 & 0 & 0 \\
\end{array}
\right).$$
In all three examples above, the matrices have the characteristic polynomial  $f_{\alpha}(t)=\left(x^4-(1-\alpha) \right)^3-\alpha ^3 x^3$, and no further edges can be added to their digraphs.  Note that the weights are determined by the fact that the matrix is required to be stochastic, and the condition that the weight on each of the three $q$-cycles has to be equal to $1-\alpha$. Finally, the free parameters $\alpha_i$ and $\alpha_i'$ have to be chosen so that the resulting matrix is nonnegative, and as we see from the realisations above, the allowed range depends on $\alpha$. 
\end{example}

\section{Type III}\label{sec:type3}

In this section we consider stochastic realisations of reduced Ito polynomials of Type III,  $f_{\alpha}(t)=t^{y}(t^q-\beta)^{d}-\alpha^d$, where  $n=q d+y$, $d \geq 2$ and $y \in \{0,\ldots, q-1\}$. Throughout this section, let $A$ be a stochastic matrix with the characteristic polynomial $f_{\alpha}$, and let $\Gamma$ be the associated directed graph. Our first aim is to determine digraphs with minimal number of edges, i.e. the sparsest realisations of $f_{\alpha}$.  

\begin{theorem}\label{thm:TIII}
Let  $f_{\alpha}(t)=t^{y}(t^q-\beta)^{d}-\alpha^d$ be a reduced Ito polynomial of Type III, $n=q d+y$, $d \geq 2$ and $y \in \{0,\ldots, q-1\}$. Let $A$ be a stochastic matrix with the characteristic polynomial $f_{\alpha}(t)$ and the minimal possible number of nonzero elements.
Then there exists a partition  $y=\sum_{k=1}^d{y_k}=y$,  $y_k \in \{0,1,\ldots,q-1\}$, so that $A$
 is, up to permutation similarity, equal to a matrix of the form: 
$$A=DC_n+(I_n-D)C_n^{n+1-q},$$
where $D \in \R^{q \times q}$ is a diagonal matrix with elements in the $(k q+\sum_{i=1}^{k} y_i)_n$-th, $k=1,\ldots,d$, positions equal to $\alpha,$ and all other diagonal elements equal to $1$. 
\end{theorem}

\begin{proof}
Let $A$ a be stochastic matrix with the characteristic polynomial $f_{\alpha}(t)$ and the minimal possible number of nonzero elements, and let $\Gamma$ be the associated digraph. Let us define:
\begin{align*}
&\mc E_n:=\{(i,1+(i)_n); i=1,\ldots,n\}, \\ 
&\hat{\mc E}_q:=\{(i,1+(i-q)_n);i=1,\ldots,n\}.
\end{align*}
By  Theorem \ref{thm:dd}, $\Gamma$ is isomorphic to a graph $\Gamma_0$ satisfying $E(\Gamma_0) \subseteq  \mc E_n \cup  \hat{\mc E}_q$. Since  $\Gamma$ contains an $n$-cycle by Proposition \ref{prop:cycle_structure}, we also have $ \mc E_n \subseteq E(\Gamma_0)$.  We still need to determine the minimal number of edges from  $\hat{\mc E}_q$ that can be included in $\Gamma$. 

Note that each edge from $\hat{\mc E}_q$ together with selected edges from $\mc E_n$ forms a $q$-cycle in $\Gamma$.  
Since  any realisation has at least $d$ $q$-cycles by Corollary \ref{cor:q_cycles}, we know that at least $d$ edges from $\hat{\mc E}_q$  will be included in $\Gamma_0$. 
Observe that there is a realisation whose associated digraph is $\mc E_n \cup \{(iq, (i-1)q+1); i=1, \ldots, d   \}$; this fact can be established by applying Theorem \ref{thm:coefficients} to the realisation in which each $q$-cycle has weight $\alpha$.  
 Hence, any sparsest realisation must contain exactly $d$ $q$-cycles. By Corollary  \ref{cor:q_cycles} they have to be vertex disjoint and of equal weight. 

Edges $(i_k,1+(i_k-q)_n )$ and $(i_k,1+(i_k-q)_n )$ together with edges form $\mc E_n$ form two disjoint $q$-cycles iff $d_n(i_k,i_l) \geq q$.  We deduce that the digraph with precisely $d$ $q$-cycles must satisfy 
$$E(\Gamma_0)=\mc E_n \cup \{(i_k,1+(i_k-q)_n ); k=1,\ldots,d\},$$ where $d_n(i_k,i_l) \geq q$ whenever $k\neq l$. 
 Note that for any collection $i_k \in \{1,2,\ldots,n\}$ satisfying $d_n(i_k,i_l) \geq q$ for $k\neq l$, there exist $y_k \in \{0,1,\ldots,q-1\}$ with $\sum_{k=1}^d{y_k}=y$, so that $i_k=k q+y_k$.
 The weights on the edges are uniquely determined by the fact that $A$ is a stochastic matrix, and that all the $q$-cycles have weight $(1-\alpha)$. 
 
 It is straightforward to check that any matrix of this form has the characteristic polynomial $f_{\alpha}(t)$ by direct calculation. 
\end{proof}

Some examples of minimal realisations are given below. 
\begin{example}
Let $f_{\alpha}(t)=x^3 (x^4 - (1 - \alpha))^3 - \alpha^3$ be a Type III reduced Ito polynomial. By Theorem \ref{thm:TIII}, all sparsest realisations of $f_{\alpha}(t)$ are of the form $A=DC_{15}+(I_{15}-D)C_{15}^{12},$ where $D$ is a diagonal matrix with entries in positions $4+y_1$, $8+y_1+y_2$, $12+y_1+y_2+y_3$ equal to $\alpha$, and all other entries equal to $1$; here  we have $y=3=y_1+y_2+y_3$. There are four relevant partitions of $3$ into three parts: $3=0+0+3$, $3=0+1+2$, $3=0+2+1$ and $3=1+1+1$. 
The choice $y_1=0$, $y_2=0$, $y_3=3$ results in 
a stochastic matrix with the following weighted digraph: 
\begin{center}
\includegraphics[height=5.5cm,keepaspectratio]{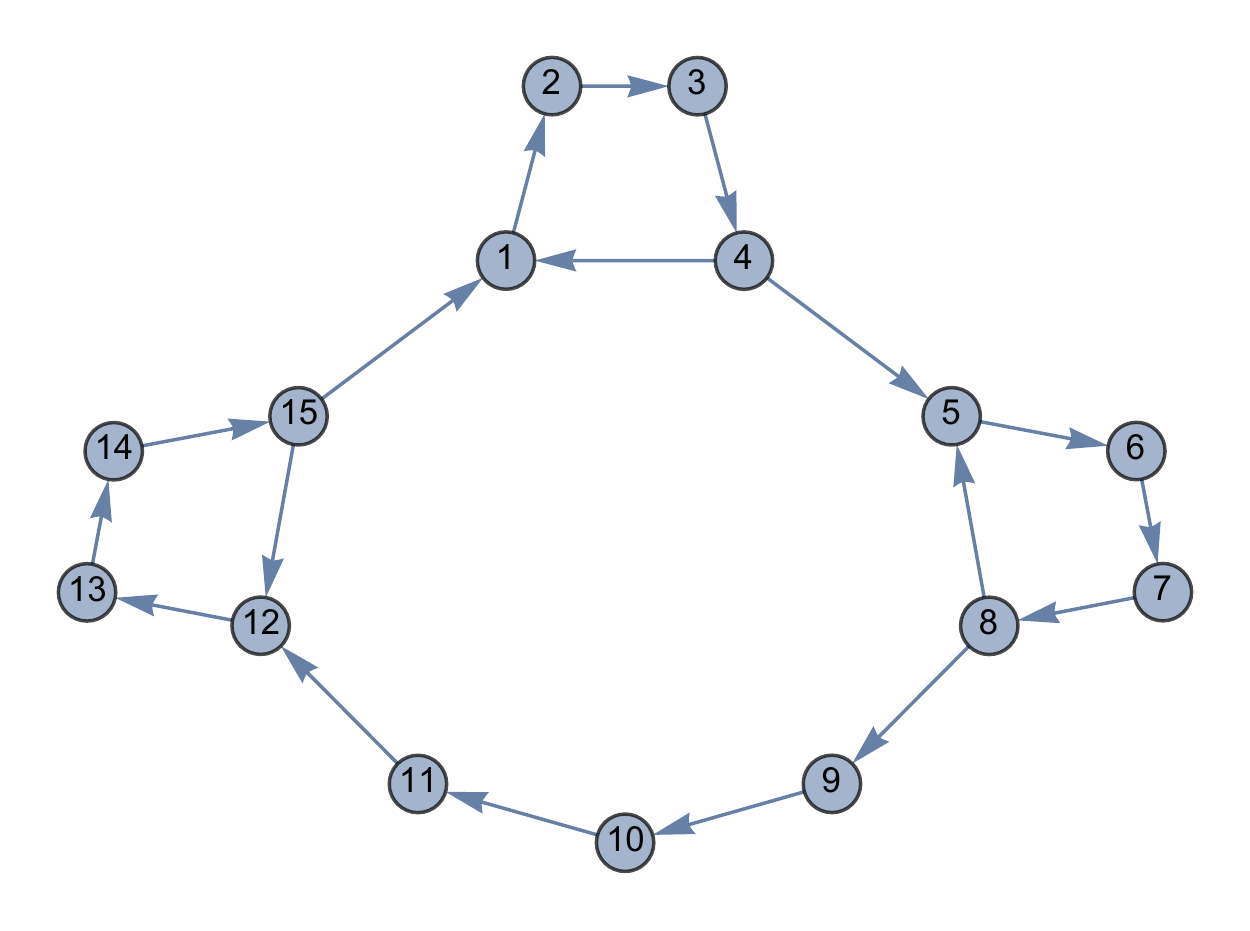}
\end{center}
The choice $y_1=0$, $y_2=1$, $y_3=2$, gives us a stochastic matrix 
whose weighted digraph is: 
\begin{center}
\includegraphics[height=5.5cm,keepaspectratio]{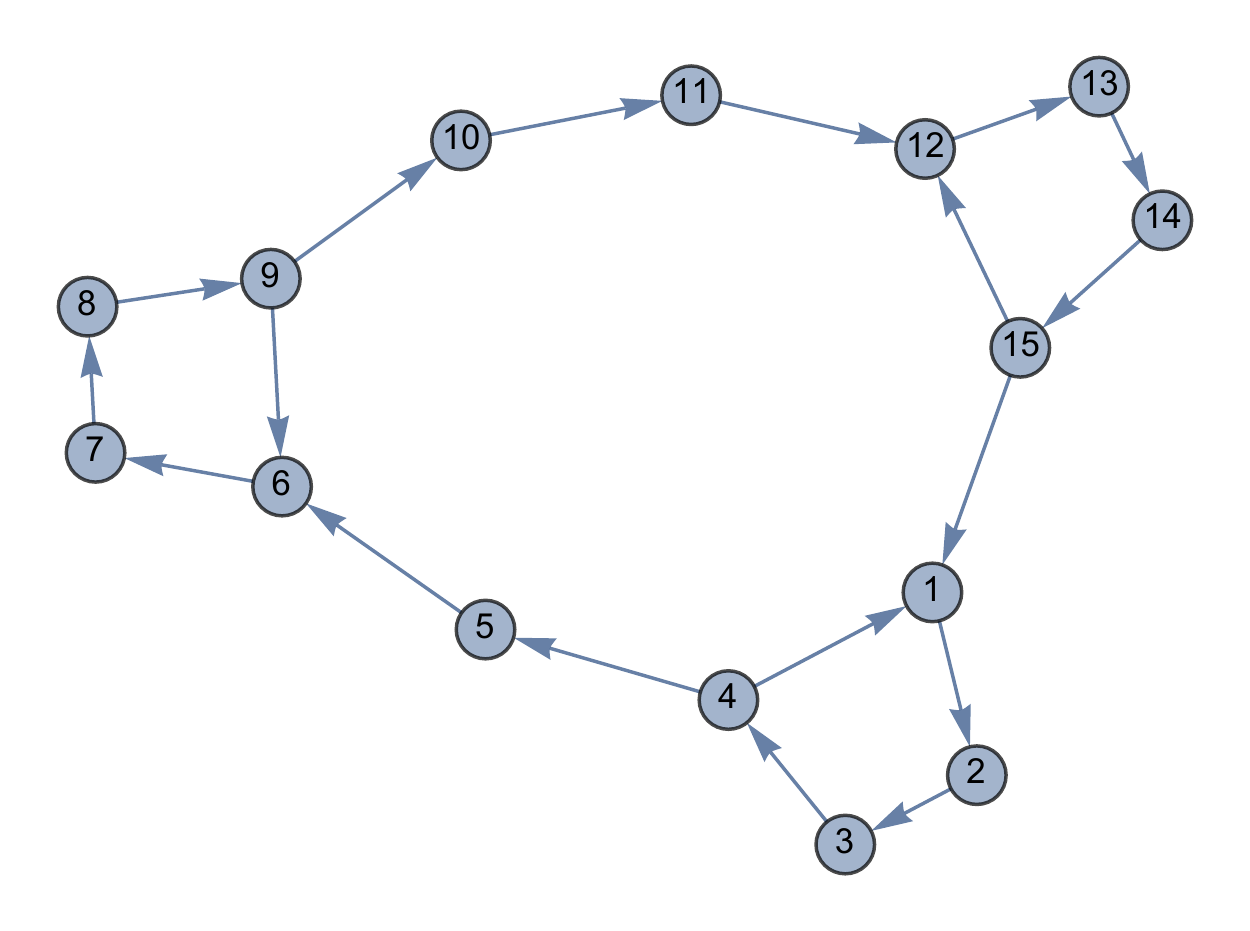}
\end{center}
The choice $y_1=0$, $y_2=2$, $y_3=1$ yields a stochastic matrix 
with weighted digraph:
\begin{center}
\includegraphics[height=5.5cm,keepaspectratio]{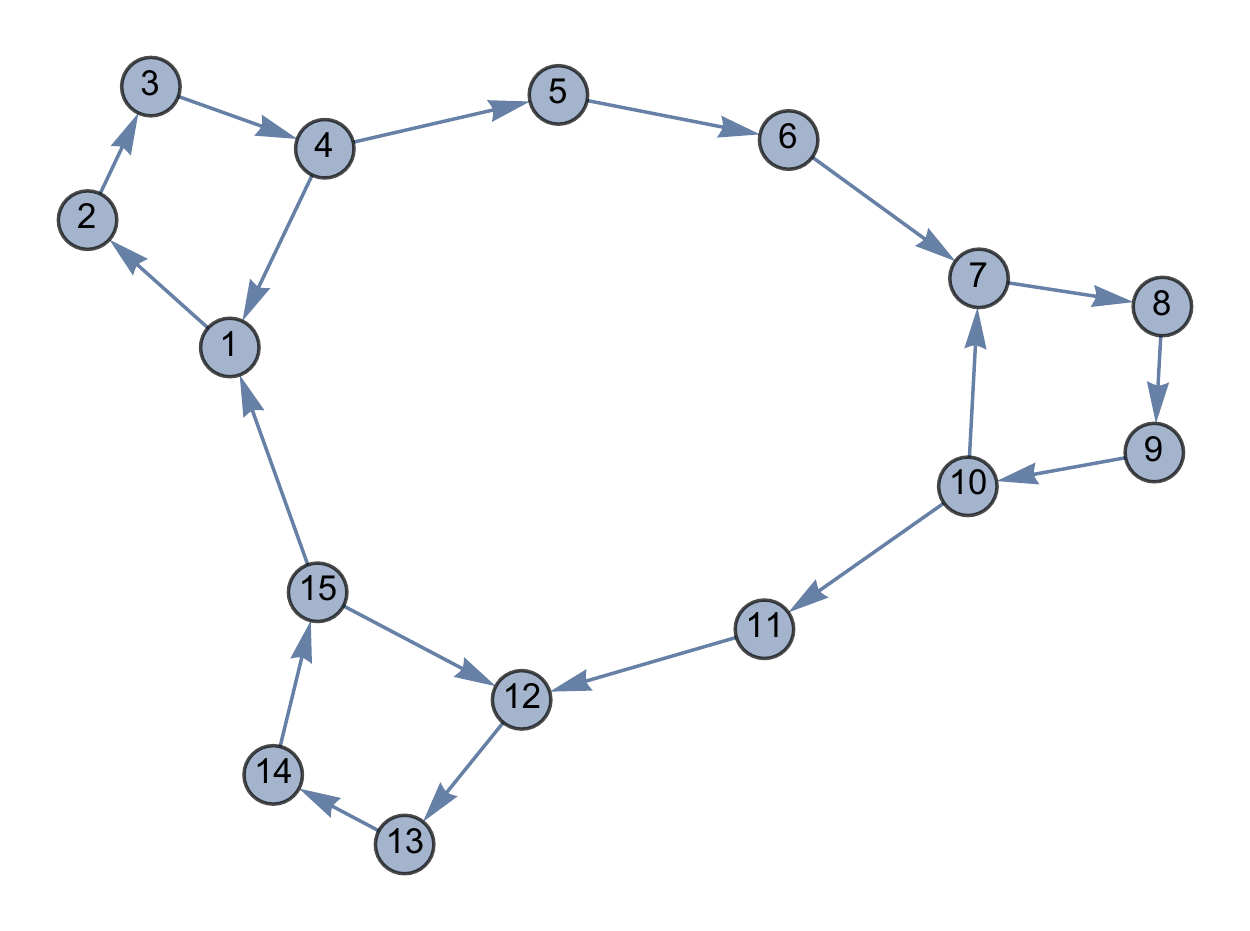}
\end{center}
Finally, the last partition yields $y_1=y_2=y_3=1.$ The directed graph of the  corresponding matrix realisation is: 
\begin{center}
\includegraphics[height=5.5cm,keepaspectratio]{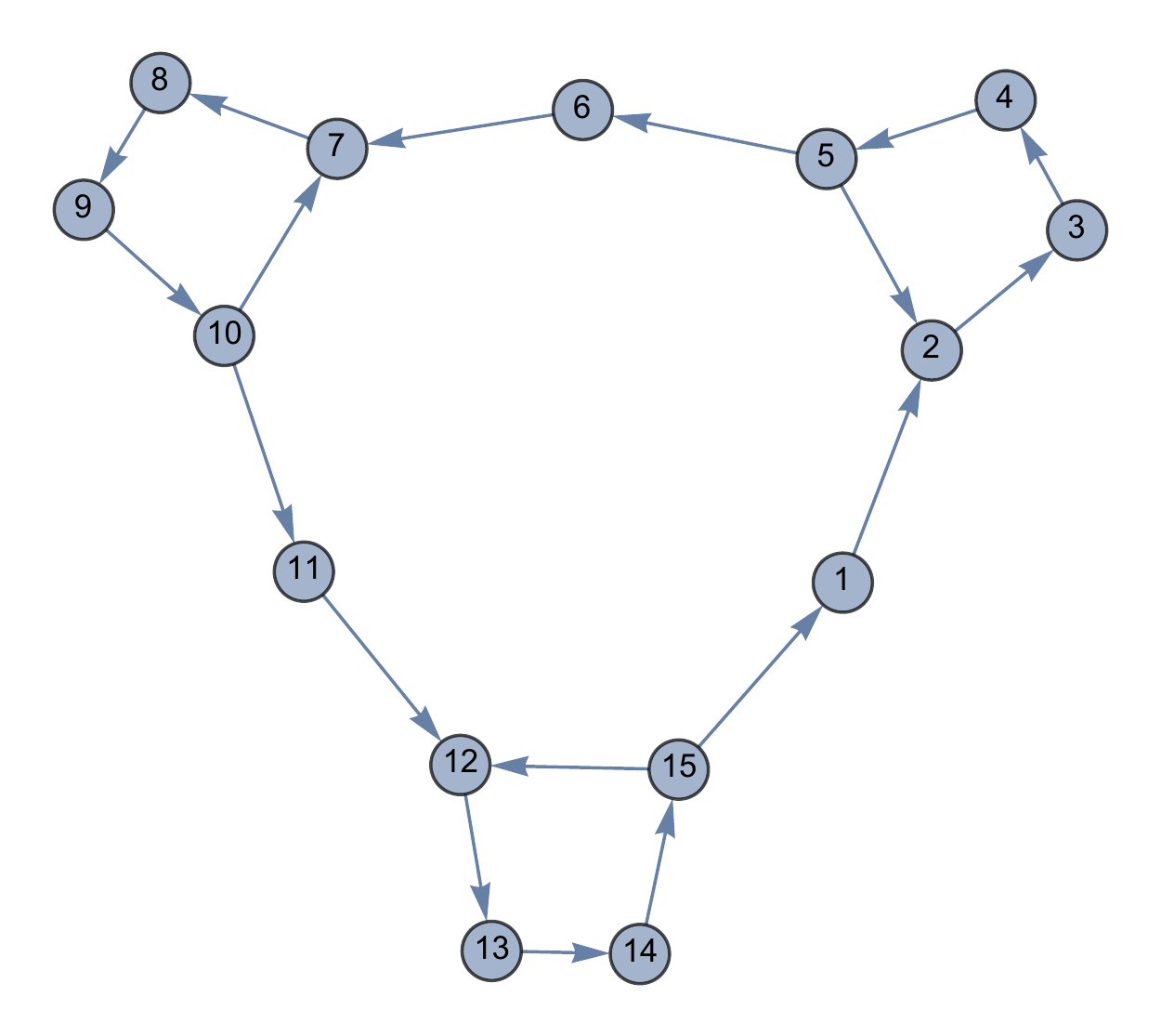}
\end{center}
\end{example}

The next proposition, assumes notation in the above proof, and presents a family of stochastic realisations of Type III polynomials.

\begin{proposition}\label{prop:type3} 
Let $n=qd+y$, where $y \in \{1,\ldots,q-1\}$. Let $V(\Gamma)=\{1,2,\ldots,n\}$, and $E(\Gamma)=\mc E_n \cup (\cup_{t=1}^d \mc B_t)$,where $\mc B_t \subseteq  \hat{\mc E}_q$ satisfy:
\begin{itemize}
\item if $(i,(i+1-q)_n),(j,(j+1-q)_n) \in \mc B_t$ then $d_n(i,j)<q$, and 
\item if $(i,(i+1-q)_n) \in \mc B_t$ and $(j,(j+1-q)_n) \in \mc B_s$, $t \neq s$, then $d_n(i,j)\geq q$.  
\end{itemize}
Let $A$ be a stochastic matrix, with the associated digraph $\Gamma$, and weights $a_{i(i+1)_n}=\alpha_i$ that satisfy:
\begin{equation}
\prod_{\{i, (i,(i+1-q)_n) \in \mc B_t\}} \alpha_i=\alpha.
\end{equation}

Then the characteristic polynomial of $A$ is equal to  $f_{\alpha}(t)=t^y(t^q-(1-\alpha))^d-\alpha^d$.  \end{proposition}

 \begin{proof}
 Let $A$ be a stochastic matrix as defined in the statement. To determine the coefficients $k_i$ of the characteristic polynomial of $A$: $$\det(x I_n-A)=x^n+\sum_{i=1}^n k_i x^{n-i},$$
   we will use Theorem \ref{thm:coefficients}; hence we need to understand the cycles in $\Gamma$.  First we note that $\Gamma$ contains an $n$-cycle, $|\cup_{t=1}^d \mc B_t|$ $q$-cycles, and no other cycles. Indeed, each edge of the form $(i,(i+1-q)_n)$ lies on a single $q$-cycle, and we will denote this $q$-cycle by $\Gamma_i$. Further, let us denote $\mc I_t:=\{i; (i,(i+1-q)_n) \in \mc B_t\}$.
 
 Clearly, $k_i=0$, if $i\neq n$ and $i$ is not divisible by $q$. Further, $k_n=\prod_{i=1}^n \alpha_i=\prod_{t=1}^d (\prod_{i \in \mc I_t} \alpha_i)=\alpha^d.$ To determine $k_{s q}$, $s=1,\ldots, d$,  we note that $V(\Gamma_i) \cap V(\Gamma_j) \neq \emptyset$ if and only if $i,j \in \mc I_t$, and we claim that $\pi_t:=\sum_{i \in \mc I_t} \pi(\Gamma_i)=1-\alpha$. 
To show this, let us without loss of generality assume that $\mc I_t\subseteq \{1,\ldots,i_0\},$ where $1 \leq i_0\leq q$. Then 
\begin{align*}
\sum_{i \in \mc I_t} \pi(\Gamma_i)&=(1-\alpha_1)+\alpha_1(1-\alpha_2)+\alpha_1\alpha_2(1-\alpha_3)+\cdots+\alpha_1\ldots \alpha_{i_0-1}(1-\alpha_{i_0}) \\
&=1-\alpha_1\alpha_2\ldots \alpha_{i_0}=1-\prod_{i \in \mc I_t} \alpha_i=1-\alpha.
\end{align*}
Since $k_{sq}$ is the elementary symmetric polynomial in $\pi_t$, $t=1,\ldots,d$, of degree $s$, the statement follows. 
 \end{proof}

 \begin{example}
We offer two examples of realisations $A_i=D_iC_{15}+(I_{15}-D_i)C_{15}^{12}$, $i=1,2$, of $f_{\alpha}(t)=x^3 (x^4 - (1 - \alpha))^3 - \alpha^3$. Taking 
$$D_1=\left(
\begin{array}{ccccccccccccccc}
 1 & 0 & 0 & 0 & 0 & 0 & 0 & 0 & 0 & 0 & 0 & 0 & 0 & 0 & 0 \\
 0 & 1 & 0 & 0 & 0 & 0 & 0 & 0 & 0 & 0 & 0 & 0 & 0 & 0 & 0 \\
 0 & 0 & 1 & 0 & 0 & 0 & 0 & 0 & 0 & 0 & 0 & 0 & 0 & 0 & 0 \\
 0 & 0 & 0 & \alpha_1 & 0 & 0 & 0 & 0 & 0 & 0 & 0 & 0 & 0 & 0 & 0 \\
 0 & 0 & 0 & 0 & \frac{\alpha }{\alpha_1} & 0 & 0 & 0 & 0 & 0 & 0 & 0 & 0 & 0 & 0 \\
 0 & 0 & 0 & 0 & 0 & 1 & 0 & 0 & 0 & 0 & 0 & 0 & 0 & 0 & 0 \\
 0 & 0 & 0 & 0 & 0 & 0 & 1 & 0 & 0 & 0 & 0 & 0 & 0 & 0 & 0 \\
 0 & 0 & 0 & 0 & 0 & 0 & 0 & 1 & 0 & 0 & 0 & 0 & 0 & 0 & 0 \\
 0 & 0 & 0 & 0 & 0 & 0 & 0 & 0 & \alpha'_1 & 0 & 0 & 0 & 0 & 0 & 0 \\
 0 & 0 & 0 & 0 & 0 & 0 & 0 & 0 & 0 & \frac{\alpha }{\alpha'_1} & 0 & 0 & 0 & 0 & 0 \\
 0 & 0 & 0 & 0 & 0 & 0 & 0 & 0 & 0 & 0 & 1 & 0 & 0 & 0 & 0 \\
 0 & 0 & 0 & 0 & 0 & 0 & 0 & 0 & 0 & 0 & 0 & 1 & 0 & 0 & 0 \\
 0 & 0 & 0 & 0 & 0 & 0 & 0 & 0 & 0 & 0 & 0 & 0 & 1 & 0 & 0 \\
 0 & 0 & 0 & 0 & 0 & 0 & 0 & 0 & 0 & 0 & 0 & 0 & 0 & \alpha''_1 & 0 \\
 0 & 0 & 0 & 0 & 0 & 0 & 0 & 0 & 0 & 0 & 0 & 0 & 0 & 0 & \frac{\alpha }{\alpha''_1} \\
\end{array}
\right)$$ we produce $A_1$ with digraph
 \begin{center}
\includegraphics[height=5.5cm,keepaspectratio]{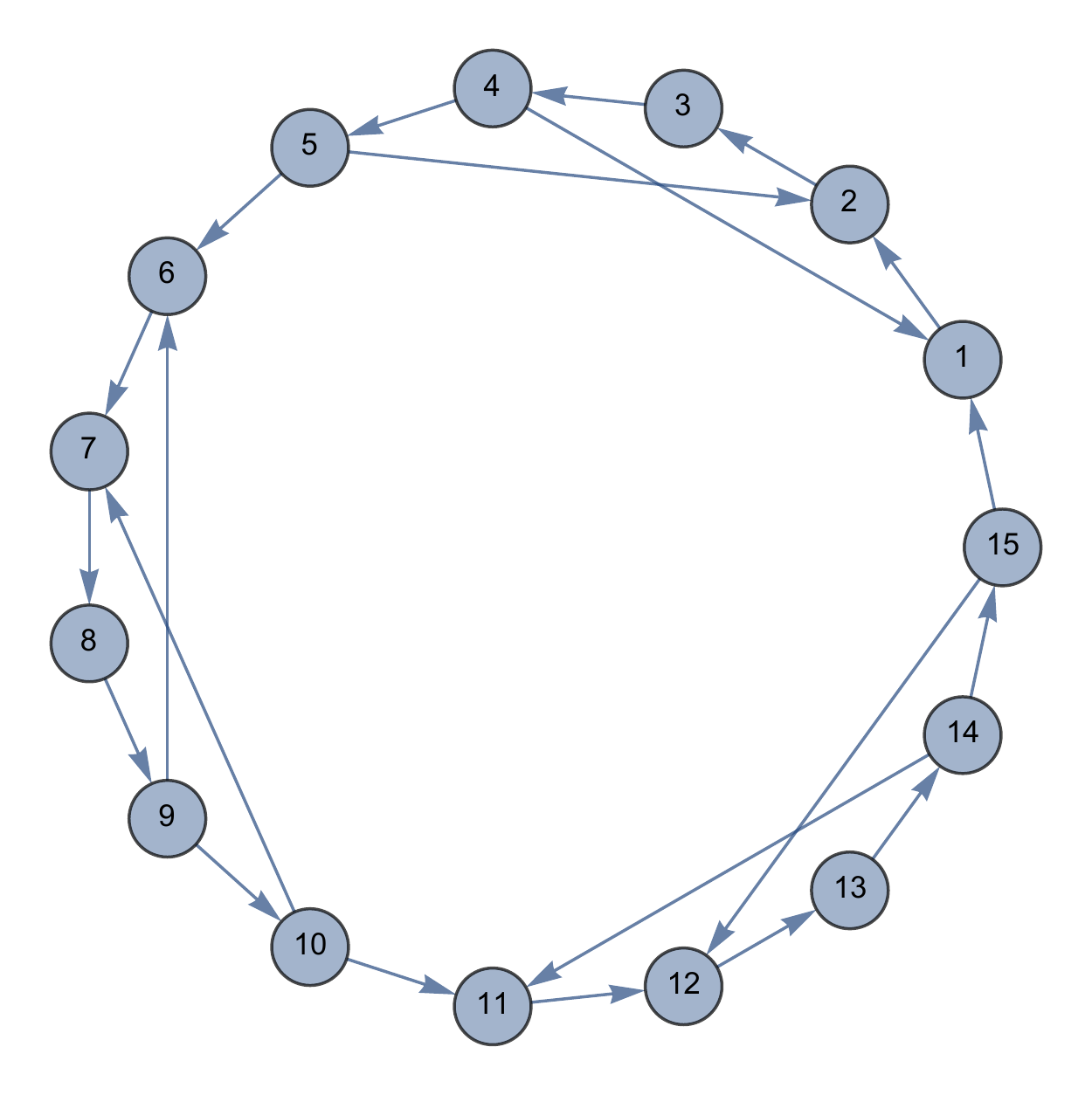}
\end{center}
Second, 
$$D_2=\left(
\begin{array}{ccccccccccccccc}
 1 & 0 & 0 & 0 & 0 & 0 & 0 & 0 & 0 & 0 & 0 & 0 & 0 & 0 & 0 \\
 0 & 1 & 0 & 0 & 0 & 0 & 0 & 0 & 0 & 0 & 0 & 0 & 0 & 0 & 0 \\
 0 & 0 & 1 & 0 & 0 & 0 & 0 & 0 & 0 & 0 & 0 & 0 & 0 & 0 & 0 \\
 0 & 0 & 0 & \alpha_1 & 0 & 0 & 0 & 0 & 0 & 0 & 0 & 0 & 0 & 0 & 0 \\
 0 & 0 & 0 & 0 & \alpha_2 & 0 & 0 & 0 & 0 & 0 & 0 & 0 & 0 & 0 & 0 \\
 0 & 0 & 0 & 0 & 0 & \alpha_3 & 0 & 0 & 0 & 0 & 0 & 0 & 0 & 0 & 0 \\
 0 & 0 & 0 & 0 & 0 & 0 & \frac{\alpha }{\alpha_1 \alpha_2 \alpha_3} & 0 & 0 & 0 &
   0 & 0 & 0 & 0 & 0 \\
 0 & 0 & 0 & 0 & 0 & 0 & 0 & 1 & 0 & 0 & 0 & 0 & 0 & 0 & 0 \\
 0 & 0 & 0 & 0 & 0 & 0 & 0 & 0 & 1 & 0 & 0 & 0 & 0 & 0 & 0 \\
 0 & 0 & 0 & 0 & 0 & 0 & 0 & 0 & 0 & 1 & 0 & 0 & 0 & 0 & 0 \\
 0 & 0 & 0 & 0 & 0 & 0 & 0 & 0 & 0 & 0 & \alpha  & 0 & 0 & 0 & 0 \\
 0 & 0 & 0 & 0 & 0 & 0 & 0 & 0 & 0 & 0 & 0 & 1 & 0 & 0 & 0 \\
 0 & 0 & 0 & 0 & 0 & 0 & 0 & 0 & 0 & 0 & 0 & 0 & 1 & 0 & 0 \\
 0 & 0 & 0 & 0 & 0 & 0 & 0 & 0 & 0 & 0 & 0 & 0 & 0 & 1 & 0 \\
 0 & 0 & 0 & 0 & 0 & 0 & 0 & 0 & 0 & 0 & 0 & 0 & 0 & 0 & \alpha  \\
\end{array}
\right)$$ produces the matrix $A_2$ 
with digraph 
 \begin{center}
\includegraphics[height=5.5cm,keepaspectratio]{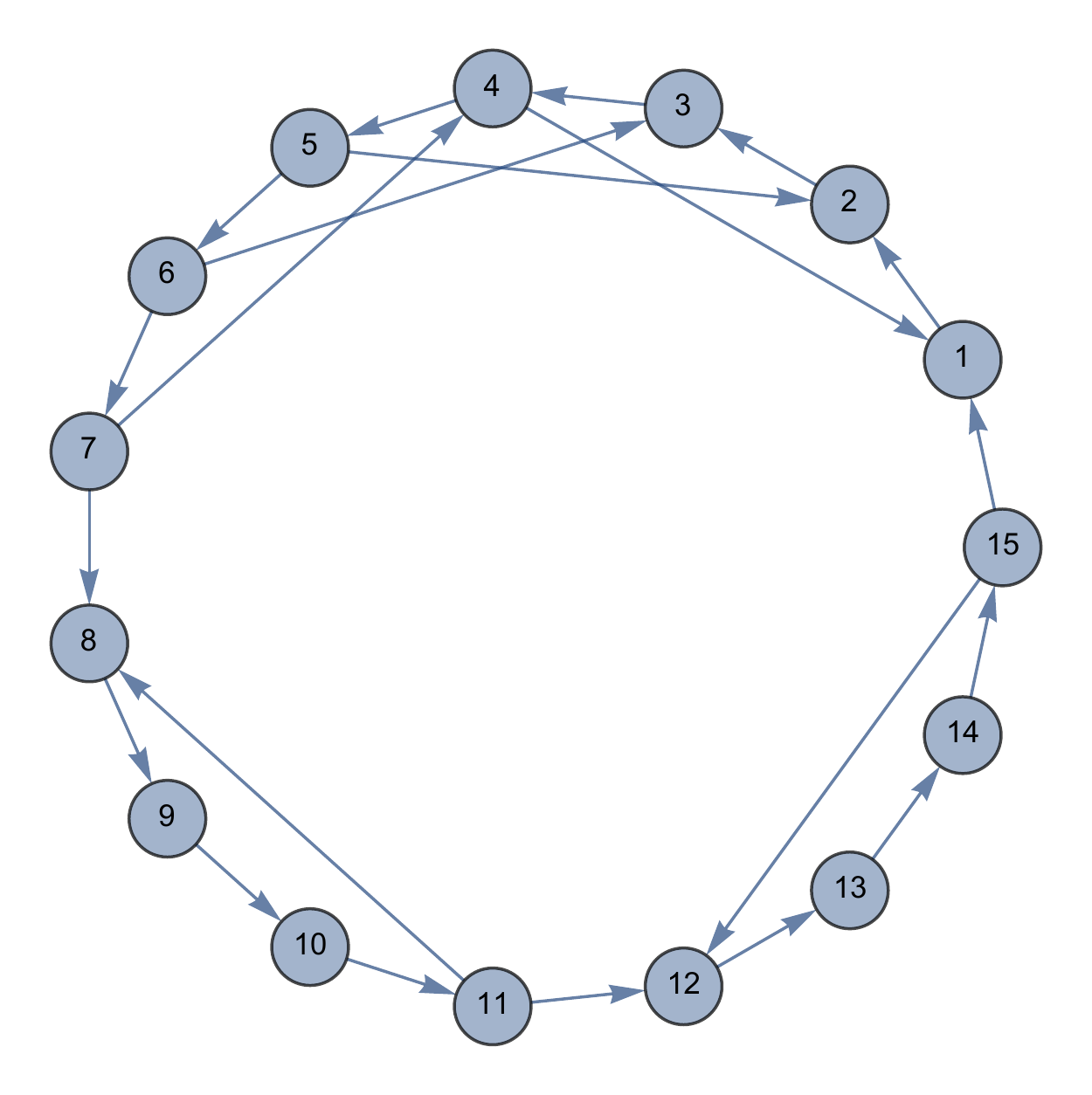}
\end{center}
The range of parameters $\alpha_i$, $\alpha'_i$, $\alpha''_i$, $\alpha'''_i$ is in both cases restricted by $\alpha$. Specifically, for $D_1$ we have $\alpha < \alpha_1,  \alpha_1', \alpha_1'' <1,$ while for $D_2$ we have $0< \alpha_1, \alpha_2, \alpha_3 <1$ and $\alpha< \alpha_1, \alpha_2, \alpha_3.$ 
 \end{example}

 We close the paper with the following. 
 
 \begin{conjecture} 
 	Let $A$ be a stochastic matrix of order $n$ whose characteristic polynomial is Type III: $f_{\alpha}(t)=t^y(t^q-(1-\alpha))^d-\alpha^d,$ where $n=qd+y, 1 \le y \le q-1$ and $d \ge 2.$ Then $A$ is permutationally similar to a matrix whose  directed graph $\Gamma $ satisfies the hypothesis of Proposition \ref{prop:type3}.  
\end{conjecture}

 \noindent 
 {\bf{Acknowledgements}}: The authors' work was supported by University College Dublin under Grant SF1588. S.K.'s research is supported in part by NSERC Discovery Grant RGPIN--2019--05408. Both authors are grateful to Tom Laffey for helpful conversations   during the early phase of this research.

\end{document}